\documentclass[12pt]{amsart}

\usepackage{amsmath,amssymb,amsfonts,amsthm,amsopn}
\usepackage{graphicx,tikz}
\usepackage[all]{xy}
\usepackage{amsmath,nccmath}
\usepackage{multirow, longtable, makecell, caption, array}
\usepackage{amssymb}
\usepackage{mathtools}
\usepackage{enumitem}

\makeatother

\usepackage{cellspace} %
\def\Spnr{\mathrm{Sp}(2d,\R)}
\def\Gltwonr{\mathrm{GL}(2d,\R)}

\newcommand{\tfa}{time-frequency analysis}

\newcommand{\ft}{Fourier transform}

\newcommand{\field}[1]{\mathbb{#1}}
\newcommand{\bR}{\field{R}}

\newcommand{\bC}{\field{C}}

\newtheorem{tm}{theorem}[section]
\newtheorem{lemma}[tm]{Lemma}

\newtheorem{theorem}{Theorem}[section]
\newtheorem{corollary}[theorem]{Corollary}
\newtheorem{definition}[theorem]{Definition}

\newtheorem{proposition}[theorem]{Proposition}
\newtheorem{remark}[theorem]{Remark}
\newtheorem*{theorem*}{Theorem}
\newtheorem*{corollary*}{Corollary}
\newcommand{\beqa}{\begin{eqnarray*}}
	\newcommand{\eeqa}{\end{eqnarray*}}

\def\la{\lambda}

\def\cF{\mathcal{F}}
\def\cS{\mathcal{S}}

\def\cB{\mathcal{B}}

\def\cU{\mathcal{U}}

\def\rd{\bR^d}

\def\rdd{{\mathbb{R}^{2d}}}

\def\R{\right)}

\def\<{\left<}
\def\>{\right>}

\def\mv1{M_v^1}

\def\Mmpq{M_m^{p,q}}
\def\phas{(x,\omega )}

\def\o{\omega}

\def\R{\mathbb{R}}
\def\Ren{\mathbb{R}^d}
\def\Renn{\rdd}

\def\sch{\mathcal{S}}

\def\Fur{\mathcal{F}}

\def\Sn2{S_{2}(L^{2}(\Ren))}
\def\S1{S_{1}(L^{2}(\Ren))}
\def\sig00{\sigma_{0,0}}

\def\la{\langle}
\def\ra{\rangle}

\newcommand{\GLL}{\mathrm{GL}\left(2d,\mathbb{R}\right)}

\begin{document}

\title[Linear perturbations of the Wigner distribution]{Linear perturbations of the Wigner distribution and the Cohen class}
	\author{Elena Cordero}
\address{Dipartimento di Matematica, Universit\`a di Torino, Dipartimento di
	Matematica, via Carlo Alberto 10, 10123 Torino, Italy}
\email{elena.cordero@unito.it}
\thanks{}

\author{S. Ivan Trapasso}
\address{Dipartimento di Scienze Matematiche, Politecnico di Torino, corso
	Duca degli Abruzzi 24, 10129 Torino, Italy}
\email{salvatore.trapasso@polito.it}
\thanks{}


\maketitle

\begin{abstract}The Wigner distribution is a milestone of Time-frequency Analysis. In order to cope with its drawbacks while preserving the desirable features that made it so popular, several kind of modifications have been proposed. This contributions fits into this perspective. We introduce a family of phase-space representations of Wigner type associated with invertible matrices and explore their general properties. As main result, we provide a characterization for the Cohen's class \cite{cohen gdist 66, cohen tfa 95}. This feature suggests to interpret this family of representations as linear perturbations of the Wigner distribution. We show which of its properties  survive under linear  perturbations and which ones are truly distinctive of its central role. 
\end{abstract}

\section{Introduction}
One of the major problems in Signal Analysis is the search for the best possible description of signals' features in terms of their pattern in time or frequency domain. It turns out that looking separately at these aspects is like taking front-view and side-view pictures of an object. Indeed, due to the ubiquitous presence of the uncertainty principle, the more accurate is the account on time evolution, the less can be said about the spectral one. This unavoidable issue can be effectively approached by jointly using both variables in order to get a faithful portrait of the signal's properties. This is in fact the paradigm of Time-frequency analysis, whose success is proven by the vast literature which has been developing from theoretical and applied problems, see \cite{cohen tfa 95,Grochenig_2001_Foundations,hlaw book} and the references therein. 

A relevant instrument for both purposes is the Wigner transform, which is defined for any $f,g\in L^2(\rd)$ as
\[
W(f,g)(x,\omega)=\int_{\rd} e^{-2\pi i y\cdot \omega}f\left(x+\frac{1}{2}y\right)\overline{g\left(x-\frac{1}{2}y\right)}dy,\qquad (x,\omega)\in\rdd.
\] 
Even if its appearance is not much revealing, the central role of this representation follows from the large number of desirable properties it satisfies. For a complete account we refer to the textbooks \cite{deGossonWigner,Grochenig_2001_Foundations,WongWeylTransform1998}. Properties of the Wigner transform are also found in \cite{FH,H}. On the other hand, again due to the multi-faceted consequences of the uncertainty principle, there is a theoretical inviolable edge surrounding the ideal time-frequency distribution: one needs to acknowledge that certain properties, though looking very natural, are mutually incompatible. For instance, in view of the physical interpretation of a phase-space distribution as signal's energy density in time-frequency space, the lack of positivity of the Wigner transform and results like Hudson's Theorem (cf. \cite{hudson 74,janssen huds 84}) raise serious concerns about the reasonable interpretation of its output. 

In order to fix this issue while retaining the good properties, smoothing the Wigner representation by means of convolution with a suitable temperate distribution $\sigma\in\cS'(\rdd)$ seemed a good compromise: the time-frequency transformations of the form 
\[
Q(f,g)=W(f,g)*\sigma, \quad f,g\in\cS(\rd),\]
are said to belong to the Cohen's class, cf. \cite{cohen gdist 66,cohen tfa 95,cohen weyl 12, Grochenig_2001_Foundations,janssen posspread 97}. There is plenty of results relating the properties of $Qf$ to suitable conditions on the Cohen's kernel $\sigma$, but one still has to deal with compatibility conditions (see the discussion in \cite[Sec. 2.5]{janssen posspread 97}). 
Within the Cohen's class, the so-called $\tau$-Wigner distributions deserve a special mention. Mimicking the definition of Weyl transform, one can introduce a family of time-frequency representations, depending on the parameter $\tau \in  [0,1]$, as follows: for any $f,g\in L^2(\rd)$:
\begin{equation}\label{wignertau}
W_{\tau }(f,g)(x,\omega )=\int_{\mathbb{R}^{d}}e^{-2\pi iy\cdot\zeta }f(x+\tau y)%
\overline{g(x-(1-\tau )y)}\, dy, \qquad (x,\omega)\in\rdd.
\end{equation}
We recapture the Wigner transform for $\tau=1/2$. These distributions have been investigated in several aspects, cf. for example \cite{boghuds, bogetal,cdet18,cnt18,janssen posspread 97}. They are  members of the Cohen's class, with a chirp-like kernel given by (cf. \cite[Proposition 5.6]{bogetal}):
\begin{equation}\label{kerneltau}
\sigma_{\tau}\left(x,\omega\right)=\frac{2^{d}}{\left|2\tau-1\right|^{d}}e^{2\pi i\frac{2}{2\tau-1}x\cdot \omega}, \quad (x,\omega)\in\rdd, \quad \tau \in [0,1]\setminus \left\{ \frac{1}{2} \right\}.
\end{equation}
It comes not as a surprise that several properties of the Wigner distribution still hold true in this context. We could meaningfully rephrase this statement by interpreting $\tau$ as a perturbation parameter and saying that these properties are stable under perturbations.

This observation effectively represents the spirit of this contribution. Inspired by the $\tau$-Wigner transforms and by the perturbative approach, we are first lead to introduce bilinear distributions of Wigner type associated with matrices, such as 
\begin{equation}\label{BAe}
\mathcal{B}_{A}\left(f,g\right)\left(x,\omega\right)=\int_{\mathbb{R}^{d}}e^{-2\pi i\omega y}f\left(A_{11}x+A_{12}y\right)\overline{g\left(A_{21}x+A_{22}y\right)}dy,\quad (x,\omega)\in\rdd,
\end{equation}

 where $A=\left(\begin{array}{cc}
A_{11} & A_{12}\\
A_{21} & A_{22}
\end{array}\right)$ is a $2d\times 2d$ invertible matrix. 
For $f=g$, we simply write $\mathcal{B}_{A}f$. 

Representations of this type have already been investigated, see e.g. \cite{bayer,bco quadratic,toft bil 17}, and indeed we limit ourselves to collect and occasionally prove a few results of general interest. Rather, the core of this work lies in the relation with the Cohen's class, as expressed by the following result. 

\begin{theorem}\label{maint}
Let $A\in \bR^{2d\times 2d}$ be an invertible matrix. The distribution $\mathcal{B}_{A}$ belongs to the Cohen's class if and only if\textbf{ $A$ }has the following special form: 
\begin{equation}\label{AM}
A=A_{M}=\left(\begin{array}{cc}
I & M+(1/2)I\\
I & M-(1/2)I
\end{array}\right),
\end{equation}
where $I$ is the $d\times d$ identity matrix and  $M\in\mathbb{R}^{d\times d}$. Furthermore, in this case we have 
\begin{equation}\label{Bamtheta}
\mathcal{B}_{A_{M}}\left(f,g\right)=W\left(f,g\right)*\theta_{M},
\end{equation}
where the Cohen's kernel $\theta_{M}\in \cS'(\rdd)$ is given by
\begin{equation}\label{thetaM}
\theta_{M}=\cF_\sigma \chi_M,\quad \mbox{with}\quad \chi_M(\xi,\eta)=e^{2\pi i \eta\cdot M\xi},
\end{equation}
i.e., the symplectic Fourier transform (cf. \eqref{SyFT} below) of the chirp-like function $\chi_M$.
\end{theorem}
We say that $A=A_{M}$ is
a {\bf Cohen-type matrix associated with $M\in\mathbb{R}^{d\times d}$}.

If $M$ is invertible, then  the kernel $\theta_M$  can be computed explicitly as
\begin{equation}\label{thM}
\theta_{M}\left(x,\omega\right)=\frac{1}{\left|\det M\right|}e^{2\pi ix\cdot M^{-1}\omega},
\end{equation}
(cf. Theorem \ref{intro2} below). 
Therefore, we are able to completely characterize a subfamily of the Cohen's class, in fact a very special one: \emph{its members can be meaningfully designed as linear perturbations of the Wigner distribution, their Cohen's kernel being  non-trivial chirp-like functions parametrized by $M$}. 
In particular, choosing $M=\left( \tau-1/2\right)I$, with $\tau\in [0,1]\setminus \{1/2\}$, we recapture the $\tau$-kernels in \eqref{kerneltau}.

These results are completely new  in the necessity part, whereas the sufficiency conditions widely extend the assumptions in \cite[Theorem 1.6.5]{bayer}. Indeed, the proof given here is quite different and allows to drop many restrictive hypotheses. 

In order to concretely unravel the effect of the perturbation matrix, in Lemma \ref{BAf} below we compute explicitly $\mathcal{B}_{A_{M}}\varphi_{\lambda}$, with $\varphi_{\lambda}\left(t\right)=e^{-\pi t^{2}/\lambda}$, $\lambda>0$.

The remaining parts of the paper are devoted to a thorough study of these phase-space transforms, always pointing at the comparison with the Wigner distribution. In particular, we show that most of its beautiful properties are preserved - rather, they are stable under linear perturbations, see Proposition \ref{bam prop}. On the other hand, the exceptional role of the Wigner and $\tau$-Wigner distributions stands out from the other representations (cf. Sec. 4.1.1).  

We then  study  the properties of the kernels $\theta_{M}$  in the framework of  modulation and Wiener amalgam spaces (cf. Section $2$ below). In line with intuition, we shall show that linear perturbations are time-frequency representations sharing the same smoothness and decay as the Wigner transform. Namely, 

\begin{theorem}
	According to the notation of  Theorem \ref{maint}, if $M\in\bR^{d\times d}$ is invertible, then
	\[
	\theta_{M}\in M^{1,\infty}\left(\rdd\right)\cap W\left(\mathcal{F}L^{1},L^{\infty}\right)\left(\rdd\right).
	\]
	Furthermore, let $f\in\mathcal{S}'\left(\mathbb{R}^{d}\right)$
	be a signal. Then, for $1\leq p,q\leq\infty$, we have 
	\[
	Wf\in M^{p,q}\left(\rdd\right) \Longleftrightarrow\mathcal{B}_{A_{M}}f\in M^{p,q}\left(\rdd\right).
	\]
	
\end{theorem}

The condition  $\theta_M \in W(FL^1,L^\infty)$ is quite natural, since it implies  the boundedness of Fourier multipliers on modulation spaces and the corresponding applications to PDE's (see the pioneering works \cite{grafakos,benyi}).

If $M$ is not invertible, then the statements of the previous result  are not valid in general. Indeed, as simple example, consider $M=0_d$, then the chirp-like function reduces to $\chi_M\equiv1$ and the related Cohen's kernel is given by $\theta_M=\cF_\sigma \chi_M=\delta$. Now, we have $\delta  \in M^{1,\infty}\left(\rdd\right)\setminus W\left(\mathcal{F}L^{1},L^{\infty}\right)\left(\rdd\right)$, cf. \cite[page 14]{cdgn tfa bj}.

An intriguing aspect that has been taken into account concerns the role of interferences. The emergence of unwanted artefacts is a well-known drawback of any quadratic representation and poses a serious problem for practical  purposes. In order to circumvent these effects as much as possible, a number of alternative distributions and damping solutions have been proposed, cf. \cite{cohen tfa 95,hlaw book,hlaw qtf} for a comprehensive discussion. Unfortunately, linear perturbations of the Wigner distribution do not result in an effective damping of interference effects. A simple toy model inspired by the discussions in \cite{bco quadratic,bogetal} shows that the effect of  perturbation consists of distortion and relocation of cross terms. In fact, this is not surprising given that the effective damping of interferences is somewhat related to the global decay of the Cohen's kernel, while $\theta_{M}\in L^{\infty}(\rdd)$. We suggest that convolution with suitable decaying distributions may provide an improvement, but the concrete risk is to loose other desirable properties. 

To conclude, we characterize boundedness of $\cB_{A_M}$ on Lebesgue, modulation and Wiener amalgam spaces. By extending known results for the Wigner distribution, we show that the continuity on these functional spaces is indeed a stable property under perturbation. 

The paper is organized as follows. In Section 2 we collect basic results of Time-frequency Analysis, essentially to fix the notation. In particular, we review the fundamental properties of modulation and Wiener amalgam spaces, but also of bilinear coordinate transformations and partial Fourier transform. In Section 3 we introduce distributions of Wigner type associated with invertible matrices in full generality and prove their relevant properties. In Section 4 we specialize to the Cohen's class and completely characterize the most important time-frequency features of the distributions arising as linear perturbations of the Wigner transform. 

\section{Preliminaries}

\textbf{Notation.} We define $t^2=t\cdot t$, for $t\in\rd$, and
$xy=x\cdot y$ is the scalar product on $\Ren$. The Schwartz class is denoted by  $\sch(\Ren)$, the space of temperate distributions by  $\sch'(\Ren)$.   The brackets  $\la f,g\ra$ denote the extension to $\sch' (\Ren)\times\sch (\Ren)$ of the inner product $\la f,g\ra=\int f(t){\overline {g(t)}}dt$ on $L^2(\Ren)$ - the latter being conjugate-linear in the second entry. The conjugate exponent $p'$ of $p \in [1,\infty]$ is defined by $1/p+1/p'=1$.

The Fourier transform of a function $f$ on $\rd$ is normalized as
\[
\Fur f(\omega)= \int_{\rd} e^{-2\pi i x\omega} f(x)\, dx,\qquad \omega \in \rd.
\]

For any $x,\omega \in \rd$, the modulation $M_{\omega}$ and translation $T_{x}$ operators are defined as 
\[
M_{\omega}f\left(t\right)= e^{2\pi it \omega}f\left(t\right),\qquad T_{x}f\left(t\right)= f\left(t-x\right).
\]
Their composition $\pi\phas=M_\omega T_x$ is called a time-frequency shift.

Given a complex-valued function $f$ on $\rd$, the involution $f^{*}$ is defined as $$f^*(t)\coloneqq \overline{f(-t)}, \quad t\in\rd.$$ 

Recall that the short-time Fourier transform of a signal $f\in\cS'(\rd)$ with respect to the window function $g \in \cS(\rd)$ is defined as
\begin{equation}\label{STFTdef}
V_gf\phas=\langle f,\pi\phas g\rangle=\Fur (fT_x g)(\omega)=\int_{\Ren}
f(y)\, {\overline {g(y-x)}} \, e^{-2\pi iy \o }\,dy.
\end{equation}
 It is not difficult to derive the fundamental identity of time-frequency analysis \cite[pag. 40]{Grochenig_2001_Foundations}:
\begin{equation}\label{FI}
V_{g}f\left(x,\omega\right)=e^{-2\pi ix\omega}V_{\hat{g}}\hat{f}\left(\omega,-x\right).
\end{equation}

In the following sections we will thoroughly work with invertible
matrices, namely elements of the group 
\[
\GLL=\left\{ M\in\mathbb{R}^{2d\times2d}\,|\,\det M\ne0\right\} .
\] 
We employ the following symbol to denote the transpose of an inverse matrix:

\[
M^{\#}\equiv (M^{-1})^{\top} = (M^{\top})^{-1}, \qquad M\in \GLL.
\]

Let $J$ denote the canonical symplectic matrix in $\mathbb{R}^{2d}$, namely
\[
J=\left(\begin{array}{cc}
0_{d} & I_{d}\\
-I_{d} & 0_{d}
\end{array}\right)\in\mathrm{Sp}\left(2d,\mathbb{R}\right),
\]
where  the
symplectic group $\mathrm{Sp}\left(2d,\mathbb{R}\right)$ is defined
by
$$
\Spnr=\left\{M\in\Gltwonr:\;M^{\top}JM=J\right\}.
$$
Observe that, for $z=\left(z_{1},z_{2}\right)\in\mathbb{R}^{2d}$, we have
$Jz=J\left(z_{1},z_{2}\right)=\left(z_{2},-z_{1}\right),$  $J^{-1}z=J^{-1}\left(z_{1},z_{2}\right)=\left(-z_{2},z_{1}\right)=-Jz,$ and 
$J^{2}=-I_{2d\times2d}.$

The symplectic Fourier transform $\cF_{\sigma}$ of a function $F$ on the phase space $\rdd$ is defined as 
\begin{equation}\label{SyFT}
\cF_{\sigma}F(x,\omega)= \cF F (J(x,\omega))= \cF F (\omega,-x).
\end{equation} 
Remark that this is an involution, i.e., $\cF_{\sigma}(\cF_{\sigma}F)=F$.

Recall that the tensor product of two functions $f,g:\mathbb{R}^{d}\rightarrow\mathbb{C}$
is defined as 
\[
f\otimes g:\rdd\rightarrow\mathbb{C}\,:\,\left(x,y\right)\mapsto f\otimes g\left(x,y\right)= f\left(x\right)g\left(y\right).
\]
It is easy to prove that the tensor product  $\otimes$ is a bilinear mapping from $L^{2}\left(\mathbb{R}^{d}\right)\times L^{2}\left(\mathbb{R}^{d}\right)$ into $L^{2}\left(\rdd\right)$. Furthermore, it maps
$\mathcal{S}\left(\mathbb{R}^{d}\right)\times\mathcal{S}\left(\mathbb{R}^{d}\right)$ into $\mathcal{S}\left(\rdd\right)$. The tensor product of two temperate distributions is also well defined by the following construction:  $f,g\in\mathcal{S}'\left(\mathbb{R}^{d}\right)$ is the distribution
$f\otimes g\in\mathcal{S}'\left(\rdd\right)$ acting on
any $\Phi\in\mathcal{S}\left(\rdd\right)$ as 
\[
\left\langle f\otimes g,\Phi\right\rangle =\left\langle f,\left\langle g,\Phi_{x}\right\rangle \right\rangle ,
\]
meaning that $g$ acts on the section $\Phi_{x}\left(y\right)$ and
then $f$ acts on $\left\langle g,\Phi_{x}\right\rangle \in\mathcal{S}\left(\mathbb{R}_{x}^{d}\right)$.
In particular, it is the unique distribution such that
\[
\left\langle f\otimes g,\phi_{1}\otimes\phi_{2}\right\rangle \equiv\left\langle f,\phi_{1}\right\rangle \left\langle g,\phi_{2}\right\rangle, \quad \forall \phi_{1},\phi_{2}\in\mathcal{S}\left(\mathbb{R}^{d}\right).
\]

In conclusion, recall that the complex conjugate $\overline{f}\in\mathcal{S}'\left(\mathbb{R}^{d}\right)$
of a temperate distribution $f\in\mathcal{S}'\left(\mathbb{R}^{d}\right)$
is defined by
\[
\left\langle \overline{f},\phi\right\rangle =\overline{\left\langle f,\overline{\phi}\right\rangle },\qquad\phi\in\mathcal{S}\left(\mathbb{R}^{d}\right).
\]

\subsection{Function spaces} 
Recall that $C_{0}(\rd)$ denotes the class of continuous functions on $\rd$ vanishing at infinity.

We say that a non-negative continuous function on $v:\mathbb{R}^{2d}\rightarrow (0,+\infty)$ is a \emph{weight function} if the following properties are satisfied:
 $v\left(0\right)=1$,  $v$ is even in each coordinate: $v\left(\pm z_{1},\ldots,\pm z_{2d}\right)=v\left(z_{1},\ldots,z_{2d}\right)$ and $v$ is submultiplicative: $v\left(w+z\right)\le v\left(w\right)v\left(z\right)$, for any $w,z\in\mathbb{R}^{2d}$.
Weights of particular relevance are those of polynomial type, namely
\begin{equation}\label{vs}
v_{s}\left(z\right)=\left\langle z\right\rangle ^{s}=\left(1+\left|z\right|^{2}\right)^{\frac{s}{2}},\qquad z\in\mathbb{R}^{2d},\,s\ge0.
\end{equation}
Notice that, for $s\geq 0$, the weight function  $v_s$ is equivalent to the submultiplicative weight $(1+|\cdot|)^s$, that is, there exist $C_1, C_2>0$ such that 
$$C_1 v_s(z)\leq (1+|z|)^s\leq C_2 v_s(z),\quad z\in\rdd.$$
A weight function $m$ on $\Renn$ is called  {\it
	$v$-moderate} if $ m(z_1+z_2)\leq Cv(z_1)m(z_2)$  for all $z_1,z_2\in\Renn.$ We write $\mathcal{M}_{v}$ to denote class of $v$-moderate weights.

\noindent
\textbf{Modulation spaces.}
Given a non-zero window $g\in\sch(\Ren)$, a $v$-moderate weight
function $m$ on $\Renn$ and $1\leq p,q\leq
\infty$, the {\it
	modulation space} $M^{p,q}_m(\Ren)$ consists of all tempered
distributions $f\in\sch'(\Ren)$ such that $V_gf\in L^{p,q}_m(\Renn )$
(weighted mixed-norm space). The norm on $M^{p,q}_m$ is
$$
\|f\|_{M^{p,q}_m}=\|V_gf\|_{L^{p,q}_m}=\left(\int_{\Ren}
\left(\int_{\Ren}|V_gf(x,\o)|^pm(x,\o)^p\,
dx\right)^{q/p}d\o\right)^{1/q}  \, .
$$
If $p=q$, we write $M^p_m$ instead of $M^{p,p}_m$, and if $m(z)\equiv 1$ on $\Renn$, then we write $M^{p,q}$ and $M^p$ for $M^{p,q}_m$ and $M^{p,p}_m$. In particular, $M^2=L^2$.

Then  $\Mmpq (\Ren )$ is a Banach space
whose definition is independent of the choice of the window $g$. Moreover, we recall that the class of admissible windows can be extended to $M^1_v$ (cf. \cite[Thm.~11.3.7]{Grochenig_2006_Time}). 

For $m\in \mathcal{M}_v$, modulation spaces enjoy the following inclusion properties: 
\begin{equation*}
\mathcal{S}(\mathbb{R}^{d})\subseteq M^{p_{1},q_{1}}_m(\mathbb{R}%
^{d})\subseteq M^{p_{2},q_{2}}_m(\mathbb{R}^{d})\subseteq \mathcal{S}^{\prime
}(\mathbb{R}^{d}),\quad p_{1}\leq p_{2},\,\,q_{1}\leq q_{2}.
\label{modspaceincl1}
\end{equation*}%

Note the connection $M^1=S_0$, the Feichtinger algebra, with dual space $M^\infty=S_0'$. Hence, properties stated for unweighted modulation spaces can be equally formulated by considering  the Banach Gelfand triple ($S_0$,$L^2$,$S_0'$) in place of  the standard Schwartz triple $(\cS',L^2,\cS')$, cf. \cite{CFL}. 

\noindent
\textbf{Wiener amalgam spaces.} Fix $g\in \cS(\rd) \setminus \left\{ 0 \right\} $. Given weight functions $u,w$ on $\rd$, the Wiener amalgam space $W(\Fur L^p_u,L^q_w)(\rd)$ can be concretely designed as the space of distributions $f\in\cS'(\rd)$ such that
\[
\|f\|_{W(\Fur L^p_u,L^q_w)(\rd)}:=\left(\int_{\Ren}
\left(\int_{\Ren}|V_gf(x,\o)|^p u^p(\o)\,
d\o\right)^{q/p} w^q(x)d x\right)^{1/q}<\infty  \,
\]
(obvious modifications for $p=\infty$ or $q=\infty$).
Using the fundamental identity of \tfa\, \eqref{FI}, we can write $|V_g f(x,\o)|=|V_{\hat g} \hat f(\o,-x)| = |\mathcal F (\hat f \, T_\o \overline{\hat g}) (-x)|$  and (recall $u(x)=u(-x)$)
$$
\| f \|_{{M}^{p,q}_{u\otimes w}} = \left( \int_{\rd} \| \hat f \ T_{\o} \overline{\hat g} \|_{\cF L^p_u}^q w^q(\o) \ d \o \right)^{1/q}
= \| \hat f \|_{W(\cF L_u^p,L_w^q)}.
$$
Hence the Wiener amalgam spaces under our consideration are simply the image under the \ft\, of modulation spaces
\begin{equation}\label{W-M}
\cF ({M}^{p,q}_{u\otimes w})=W(\cF L_u^p,L_w^q).
\end{equation}

This should not come as a surprise, since it is exactly how modulation spaces have been originally introduced by Feichtinger, i.e., as special Wiener amalgams on the Fourier transform side, cf. \cite{Feich2006} and the references therein for details.

From now on we tacitly assume the  results formulated for $L^2$-functions hold with equality
almost everywhere.

\subsection{Bilinear coordinate transformations}
Let us now define the bilinear coordinate transformation we are going to use in the sequel.
\begin{definition}
The bilinear coordinate transformation $\mathfrak{T}_{M}$, associated with a matrix $M\in\mathbb{R}^{2d\times2d}$
is defined as 
\[
\mathfrak{T}_{M}F\left(x,y\right)= F\left(M\left(\begin{array}{c}
x\\
y
\end{array}\right)\right),\qquad x,y\in\mathbb{R}^{d},
\]
where  $F$ is a function $F:\rdd\rightarrow\mathbb{C}$.
In particular, if $M=\left(\begin{array}{cc}
M_{11} & M_{12}\\
M_{21} & M_{22}
\end{array}\right)$ with $M_{ij}\in\mathbb{R}^{d\times d}$, $i,j=1,2$, we write 
\[
\mathfrak{T}_{M}F\left(x,y\right)=F\left(M_{11}x+M_{12}y,M_{21}x+M_{22}y\right).
\]
\end{definition}
The composition of two such coordinate transformations associated
with $M,N\in\mathbb{R}^{2d\times2d}$ yields $\mathfrak{T}_{M}\mathfrak{T}_{N}=\mathfrak{T}_{MN}$.
If the invertibility of $M$ is assumed, it is easy to prove the following
result. 
\begin{lemma}
\label{coord trans isom}
\begin{enumerate}[label=(\roman*)]
\item If $M\in\mathrm{GL}\left(2d,\mathbb{R}\right)$, the transformation
$\mathfrak{T}_{M}$ is a topological isomorphism on $L^{2}\left(\rdd\right)$
with inverse $\mathfrak{T}_{M}^{-1}=\mathfrak{T}_{M^{-1}}$ and adjoint
$\mathfrak{T}_{M}^{*}=\left|\det M\right|^{-1}\mathfrak{T}_{M^{-1}}$. 
\item If $M\in\mathrm{GL}\left(2d,\mathbb{R}\right)$, the transformation
$\mathfrak{T}_{M}$ is a topological isomorphism on $\mathcal{S}\left(\rdd\right)$,
hence uniquely extends to an isomorphism on $\mathcal{S}'\left(\rdd\right)$. 
\end{enumerate}
\end{lemma}
Two coordinate transformations deserve special notation: one is given
by the \emph{flip} operator, denoted as follows: for any $F\in L^2\left(\rdd\right)$,
\[
\tilde{F}\left(x,y\right)\equiv\mathfrak{T}_{\tilde{I}}F\left(x,y\right)=F\left(y,x\right),\qquad\tilde{I}=\left(\begin{array}{cc}
0_{d} & I_{d}\\
I_{d} & 0_{d}
\end{array}\right)\in\GLL,
\]
while the other one is the \emph{reflection} operator: 
\[
\mathcal{I}F\left(x,y\right)\equiv\mathfrak{T}_{-I}F\left(x,y\right)=F\left(-x,-y\right).
\]
Sometimes we will also write $\mathcal{I}=-I\in\GLL$,
in line with a common harmless practice. 

The following commutation relations between coordinate transformations
and  time-frequency shifts are easily derived.
\begin{lemma}
\label{tfs coortr}Let $A\in\GLL$.
For any $f\in L^{2}\left(\mathbb{R}^{d}\right)$, $x,\omega\in\mathbb{R}^{d}$:
\[
\mathfrak{T}_{A}T_{x}f=T_{A^{-1}x}\mathfrak{T}_{A}f,\qquad\mathfrak{T}_{A}M_{\omega}f=M_{A^{\top}\omega}\mathfrak{T}_{A}f,
\]
hence
\[
\mathfrak{T}_{A}\pi\left(x,\omega\right)f=\pi\left(A^{-1}x,A^{\top}\omega\right)\mathfrak{T}_{A}f.
\]
\end{lemma}

\subsection{Partial Fourier transforms}
In the sequel we shall work with partial Fourier transforms. Let us recall their definition and main properties.

\begin{definition}
Given $F\in L^{2}\left(\mathbb{R}_{\left(x,y\right)}^{2d}\right)$,
the symbols $\mathcal{F}_{1}$ and $\mathcal{F}_{2}$ denote the partial
Fourier transforms defined as follows:
$$
\mathcal{F}_{1}F\left(\xi,y\right)=\widehat{F_{y}}\left(\xi\right)=\int_{\mathbb{R}^{d}}e^{-2\pi i\xi t}F\left(t,y\right)dt,\quad 
\mathcal{F}_{2}F\left(x,\omega\right)=\widehat{F_{x}}\left(\omega\right)=\int_{\mathbb{R}^{d}}e^{-2\pi i\omega t}F\left(x,t\right)dt,
$$
where $\widehat{\cdot}$ denotes the Fourier transform on $L^{2}\left(\mathbb{R}^{d}\right)$
whereas
\[
F_{x}\left(y\right)=F\left(x,y\right),\qquad F_{y}\left(x\right)=F\left(x,y\right)
\]
are the sections of $F$ at fixed $x\in\mathbb{R}^{d}$ and $y\in\mathbb{R}^{d}$
respectively. Without further assumptions, the integral representations given above are to be intended
in a formal sense.
\end{definition}
Fubini's theorem assures that $F_{x}\in L^{2}\left(\mathbb{R}_{y}^{d}\right)$
for a.e. $x\in\mathbb{R}^{d}$ and $F_{y}\in L^{2}\left(\mathbb{R}_{x}^{d}\right)$
for a.e. $y\in\mathbb{R}^{d}$, thus $\mathcal{F}_{1}F$ and $\mathcal{F}_{2}F$
are indeed well defined. The Fourier transform $\mathcal{F}$ of $F\left(x,y\right)$
is therefore related to the partial Fourier transforms as 
\[
\mathcal{F}=\mathcal{F}_{1}\mathcal{F}_{2}=\mathcal{F}_{2}\mathcal{F}_{1}.
\]

We state the following result only for $\mathcal{F}_{2}$, since it
is the transform of our interest hereinafter. Similar claims for $\mathcal{F}_{1}$
can be proved following the same pattern with suitable modifications.
The proof is a matter of computation.

\begin{lemma}~\label{partial Fou isom}
(i) The partial Fourier transform $\mathcal{F}_{2}$ is an isometric (hence
topological) isomorphism on $L^{2}\left(\rdd\right)$.
In particular, 
\[
\mathcal{F}_{2}^{*}F\left(x,y\right)=\mathcal{F}_{2}^{-1}F\left(x,y\right)=\mathcal{F}_{2}F\left(x,-y\right)=\mathfrak{T}_{\mathcal{I}_{2}}\mathcal{F}_{2}F\left(x,y\right),
\]
where $\mathcal{I}_{2}=\left(\begin{array}{cc}
I & 0\\
0 & -I
\end{array}\right)$.\\
(ii) The partial Fourier transform $\mathcal{F}_{2}$ is a topological
isomorphism on $\mathcal{S}\left(\rdd\right)$, hence it
uniquely extends to an isomorphism on $\mathcal{S}'\left(\rdd\right)$. 
\end{lemma}
Interactions among partial Fourier transforms and coordinate transformations
or time-frequency shifts are derived in the following lemmas. 
\begin{lemma}
~\label{partial Fou ctran}Let $A\in\GLL$
and $F\in L^{2}\left(\rdd\right)$. Then\\
(i) $\mathcal{F}_{1}F\left(\xi,\omega\right)=\mathcal{F}_{2}\tilde{F}\left(\omega,\xi\right)=\widetilde{\mathcal{F}_{2}\tilde{F}}\left(\xi,\omega\right)$.\\
(ii) $\mathcal{F}_{1}\mathfrak{T}_{A}F\left(\xi,y\right)=\mathcal{F}_{2}\mathfrak{T}_{B}F\left(y,\xi\right)=\widetilde{\mathcal{F}_{2}\mathfrak{T}_{B}F}\left(\xi,y\right)$,
where 
\[
B=A\tilde{I}=\left(\begin{array}{cc}
A_{12} & A_{11}\\
A_{22} & A_{21}
\end{array}\right).
\]
\end{lemma}
\begin{lemma}
\label{tfs partial Fou}

For any $F\in L^{2}\left(\rdd\right)$, $\left(r,s\right),\left(\rho,\sigma\right)\in\rdd$, we have
\[
\mathcal{F}_{2}T_{\left(r,s\right)}F\left(x,\omega\right)=M_{\left(0,-s\right)}T_{\left(r,0\right)}\mathcal{F}_{2}F\left(x,\omega\right)=e^{-2\pi i\omega s}\mathcal{F}_{2}F\left(x-r,\omega\right),
\]
\[
\mathcal{F}_{2}M_{\left(\rho,\sigma\right)}F\left(x,\omega\right)=M_{\left(\rho,0\right)}T_{\left(0,\sigma\right)}\mathcal{F}_{2}F\left(x,\omega\right)=e^{2\pi ix\rho}\mathcal{F}_{2}F\left(x,\omega-\sigma\right).
\]
Hence
\[
\mathcal{F}_{2}\left(M_{\left(\rho,\sigma\right)}T_{\left(r,s\right)}F\right)\left(x,\omega\right)=e^{2\pi i\sigma s}M_{\left(\rho,-s\right)}T_{\left(r,\sigma\right)}\mathcal{F}_{2}F\left(x,\omega\right).
\]
\end{lemma}

\section{Distributions of Wigner type associated with invertible matrices}

We introduce here the main ingredients of this study. Our presentation is nearly identical to the one provided in \cite{bayer}, which is indeed richer than ours on general aspects. Anyway, we decided to develop here all the needed material in order to uniform the notation once for all and also to provide new results or shorter proofs whenever possible. 

\begin{definition}
\label{biltfr}Let $f,g\in L^{2}\left(\mathbb{R}^{d}\right)$
and $A=\left(\begin{array}{cc}
A_{11} & A_{12}\\
A_{21} & A_{22}
\end{array}\right)\in\GLL.$ The time-frequency distribution\textbf{ }of Wigner type for $f$
and $g$ associated with $A$ (in short: matrix-Wigner distribution,
MWD) is defined as 
\begin{equation}\label{BAi}
\mathcal{B}_{A}\left(f,g\right)\left(x,\omega\right)=\mathcal{F}_{2}\mathfrak{T}_{A}\left(f\otimes\overline{g}\right)\left(x,\omega\right),
\end{equation}
that is, formula \eqref{BAe}. When $g=f$, we simply write $\mathcal{B}_{A}f$ for $\mathcal{B}_{A}\left(f,f\right)$.
\end{definition}
This class of time-frequency representations includes some of the
most relevant distributions in time-frequency analysis, such as the the short-time Fourier transform:
\begin{equation}\label{ASTFT}
V_{g}f\left(x,\omega\right)=\int_{\rd} e^{-2\pi i\omega y}f\left(y\right)\overline{g\left(y-x\right)}dy=\mathcal{B}_{A_{ST}}\left(f,g\right)\left(x,\omega\right),\quad A_{ST}=\left(\begin{array}{cc}
0 & I\\
-I & I
\end{array}\right),
\end{equation}
and the $\tau$-Wigner distribution: for any $\tau\in\left[0,1\right]$,
\begin{equation}\label{tauwig}
W_{\tau}\left(f,g\right)\left(x,\omega\right)=\int_{\rd} e^{-2\pi i\omega y}f\left(x+\tau y\right)\overline{g\left(x-\left(1-\tau\right)y\right)}dy=\mathcal{B}_{A_{\tau}}\left(f,g\right)\left(x,\omega\right),
\end{equation}
where 
\begin{equation}\label{Atau}
A_{\tau}=\left(\begin{array}{cc}
I & \tau I\\
I & -\left(1-\tau\right)I
\end{array}\right).
\end{equation}
In particular, this parametric family of distributions includes 
\begin{itemize}
\item the Wigner(-Ville) distribution, corresponding to $\tau=1/2$:
\begin{equation}\label{wig}
W\left(f,g\right)\left(x,\omega\right)=\int_{\rd} e^{-2\pi i\omega y}f\left(x+\frac{y}{2}\right)\overline{g\left(x-\frac{y}{2}\right)}dy=\mathcal{B}_{A_{1/2}}\left(f,g\right)\left(x,\omega\right).
\end{equation}
\item the Rihaczek distribution, corresponding to $\tau=0$:
\begin{equation}\label{rihaczek}
R\left(f,g\right)\left(x,\omega\right)=\int_{\rd} e^{-2\pi i\omega y}f\left(x\right)\overline{g\left(x-y\right)}dy=e^{-2\pi ix\omega}f\left(x\right)\overline{\hat{g}\left(\omega\right)}=\mathcal{B}_{A_{0}}\left(f,g\right)\left(x,\omega\right).
\end{equation}
\item the conjugate-Rihaczek distribution, corresponding to $\tau=1$. 
\end{itemize}
Even the  cross-ambiguity distribution is  a MWD:
\begin{equation}\label{ambiguity}
Amb\left(f,g\right)\left(x,\omega\right)=\int_{\rd} e^{-2\pi i\omega y}f\left(y+\frac{x}{2}\right)\overline{g\left(y-\frac{x}{2}\right)}dy=\mathcal{B}_{A_{Amb}}\left(f,g\right)\left(x,\omega\right),
\end{equation}
where 
\[
A_{Amb}=\left(\begin{array}{cc}
\frac{1}{2}I & I\\
-\frac{1}{2}I & I
\end{array}\right).
\]
From Definition \ref{biltfr} and Lemmas \ref{coord trans isom}, and \ref{partial Fou isom}, we can immediately infer boundedness
properties of $\mathcal{B}_{A}\left(f,g\right)$ in the context of
the fundamental triple $\mathcal{S}\left(\mathbb{R}^{d}\right)\subset L^{2}\left(\mathbb{R}^{d}\right)\subset\mathcal{S}'\left(\mathbb{R}^{d}\right)$, as detailed below. 
\begin{proposition}
\label{def bilA triple} Assume $A\in\GLL$. Then,
\begin{enumerate}[label=(\roman*)]
\item If $f,g\in L^{2}(\mathbb{R}^{d})$, then $\mathcal{B}_{A}(f,g)\in L^{2}(\rdd)$
and the mapping $\mathcal{B}_{A}:L^{2}(\mathbb{R}^{d})\times L^{2}(\mathbb{R}^{d})\rightarrow L^{2}(\rdd)$
is continuous. Furthermore, $\mathrm{span}\left\{ \mathcal{B}_{A}(f,g)\,|\,f,g\in L^{2}(\mathbb{R}^{d})\right\} $
is a dense subset of $L^{2}(\rdd)$.
\item If $f,g\in\mathcal{S}(\mathbb{R}^{2d})$, then $\mathcal{B}_{A}(f,g)\in\mathcal{S}(\mathbb{R}^{d})$
and the mapping  $\mathcal{B}_{A}:\mathcal{S}(\mathbb{R}^{d})\times\mathcal{S}(\mathbb{R}^{d})\rightarrow\mathcal{S}(\rdd)$
is continuous. 
\item If $f,g\in\mathcal{S}'(\mathbb{R}^{d})$, then $\mathcal{\mathcal{B}}_{A}(f,g)\in\mathcal{S}'(\rdd)$
and the mapping  $\mathcal{B}_{A}:\mathcal{S}'(\mathbb{R}^{d})\times\mathcal{S}'(\mathbb{R}^{d})\rightarrow\mathcal{S}'(\rdd)$
is continuous. 
\end{enumerate}
\end{proposition}
Elementary properties of $\mathcal{B}_{A}\left(f,g\right)$ are the following. 
\begin{proposition}[Interchanging $f$ and $g$]
\label{interchanging f g}

Let $A\in\GLL$ and $f,g\in L^{2}\left(\mathbb{R}^{d}\right)$.
Then
\[
\mathcal{B}_{A}\left(g,f\right)\left(x,\omega\right)=\overline{\mathcal{B}_{C}\left(f,g\right)\left(x,\omega\right)},
\]
\end{proposition}
where 
\[
C=\tilde{I}A\mathcal{I}_{2}=\left(\begin{array}{cc}
0 & I\\
I & 0
\end{array}\right)\left(\begin{array}{cc}
A_{11} & A_{12}\\
A_{21} & A_{22}
\end{array}\right)\left(\begin{array}{cc}
I & 0\\
0 & -I
\end{array}\right)=\left(\begin{array}{cc}
A_{21} & -A_{22}\\
A_{11} & -A_{12}
\end{array}\right).
\]
In particular, $\mathcal{B}_{A}f$ is a real-valued function if and
only if $A=C$, namely 
\[
A_{11}=A_{21},\qquad A_{12}=-A_{22}.
\]

\begin{proof}
This is an easy computation:
\[
\mathfrak{T}_{A}\left(g\otimes\overline{f}\right)=\overline{\mathfrak{T}_{\tilde{I}A}\left(f\otimes\overline{g}\right)},
\]
\[
\mathcal{B}_{A}\left(g,f\right)=\mathcal{F}_{2}\overline{\mathfrak{T}_{\tilde{I}A}\left(f\otimes\overline{g}\right)}=\overline{\mathcal{F}_{2}\mathfrak{T}_{C}\left(f\otimes\overline{g}\right)}=\overline{\mathcal{B}_{C}\left(f,g\right)},
\]
as desired.
\end{proof}
The following is a generalization of the fundamental identity of time-frequency analysis for the STFT, cf. \eqref{FI}. 
\begin{proposition}[Fundamental-like identity of TFA]

Let $A\in\GLL$ and $f,g\in L^{2}\left(\mathbb{R}^{d}\right)$.
Then
\[
\mathcal{B}_{A}\left(\hat{f},\hat{g}\right)\left(x,\omega\right)=\left|\det A\right|^{-1}\mathcal{B}_{C}\left(f,g\right)\left(-\omega,x\right),
\]
where 
\[
C=\mathcal{I}_{2}A^{\#}\tilde{I}=\left(\begin{array}{cc}
I & 0\\
0 & -I
\end{array}\right)\left(A^{-1}\right)^{\top}\left(\begin{array}{cc}
0 & I\\
I & 0
\end{array}\right).
\]
\end{proposition}
\begin{proof}
First of all, notice that $\hat{f}\otimes\overline{\hat{g}}=\hat{f}\otimes\widehat{g^{*}}=\mathcal{F}\left(f\otimes g^{*}\right)$,
where $g^{*}\left(t\right)=\overline{g\left(-t\right)}$. Then, an
easy computation shows that 
\[
\mathfrak{T}_{A}\mathcal{F}\left(H\right)=\frac{1}{\left|\det A\right|}\mathcal{F}\mathfrak{T}_{A^{\#}}\left(H\right),\qquad\forall H\in L^{2}\left(\rdd\right),
\]
where $A^{\#}=\left(A^{-1}\right)^{\top}$. Therefore, 
\[
\mathcal{B}_{A}\left(\hat{f},\hat{g}\right)=\left|\det A\right|^{-1}\mathcal{F}_{2}\mathcal{F}\mathfrak{T}_{A^{\#}}\left(f\otimes g^{*}\right)=\left|\det A\right|^{-1}\mathcal{I}_{2}\mathcal{F}_{1}\mathfrak{T}_{A^{\#}}\left(f\otimes g^{*}\right),
\]
where we used $\mathcal{F}=\mathcal{F}_{2}\mathcal{F}_{1}$ and $\mathcal{F}_{2}^{2}=\mathcal{I}_{2}$.
Notice now that 
\[
\mathfrak{T}_{A^{\#}}\left(f\otimes g^{*}\right)=\mathfrak{T}_{\mathcal{I}_{2}A^{\#}}\left(f\otimes\overline{g}\right),
\]
where $\mathcal{I}_{2}=\left(\begin{array}{cc}
I & 0\\
0 & -I
\end{array}\right)$. In conclusion, Lemma \ref{partial Fou ctran} gives 
\[
\mathcal{F}_{1}\mathfrak{T}_{QA^{\#}}=\widetilde{\mathcal{F}_{2}\mathfrak{T}_{QA^{\#}\tilde{I}}},
\]
hence the claimed result:
\[
\mathcal{B}_{A}\left(\hat{f},\hat{g}\right)\left(x,\omega\right)=\left|\det A\right|^{-1}\widetilde{\mathcal{I}_{2}\mathcal{B}_{QA^{\#}\tilde{I}}\left(f,g\right)}\left(x,\omega\right)=\left|\det A\right|^{-1}\mathcal{B}_{QA^{\#}\tilde{I}}\left(f,g\right)\left(-\omega,x\right).
\]
\end{proof}
\begin{proposition}[Fourier transform of a BTFD]
\label{ft of btfd}

Let $A\in\GLL$ and $f,g\in L^{2}\left(\mathbb{R}^{d}\right)$.
Then, 
\begin{equation}\label{FouB}
\mathcal{F}\mathcal{B}_{A}\left(f,g\right)\left(\xi,\eta\right)=\mathcal{B}_{AJ}\left(f,g\right)\left(\eta,\xi\right)=\mathcal{B}_{A\mathcal{I}_{2}}\left(f,g\right)\left(\xi,\eta\right),
\end{equation}
where 
\[
AJ=\left(\begin{array}{cc}
A_{11} & A_{12}\\
A_{21} & A_{22}
\end{array}\right)\left(\begin{array}{cc}
0 & I\\
-I & 0
\end{array}\right)=\left(\begin{array}{cc}
-A_{12} & A_{11}\\
-A_{22} & A_{21}
\end{array}\right).
\]
\end{proposition}
\begin{proof}
Since $\mathcal{B}_{A}\left(f,g\right)\in L^{2}\left(\rdd\right)$,
we have:
\[
\mathcal{F}\mathcal{B}_{A}\left(f,g\right)\left(\xi,\eta\right)=\mathcal{F}_{1}\mathcal{F}_{2}^{2}\mathfrak{T}_{A}\left(f\otimes\overline{g}\right)\left(\xi,\eta\right)=\mathcal{F}_{1}\mathfrak{T}_{A}\left(f\otimes\overline{g}\right)\left(\xi,-\eta\right).
\]
From Lemma \ref{partial Fou ctran} we get
\[
\mathcal{F}_{1}\mathfrak{T}_{A}\left(f\otimes\overline{g}\right)\left(\xi,-\eta\right)=\mathcal{F}_{2}\mathfrak{T}_{A\tilde{I}}\left(f\otimes\overline{g}\right)\left(-\eta,\xi\right)=\mathcal{B}_{A\tilde{I}}\left(f,g\right)\left(-\eta,\xi\right).
\]
To conclude, notice that 
\[
\mathcal{B}_{A\tilde{I}}\left(f,g\right)\left(-\eta,\xi\right)=\mathcal{B}_{A\tilde{I}\mathcal{I}_{1}}\left(f,g\right)\left(\eta,\xi\right)=\mathcal{B}_{AJ}\left(f,g\right)\left(\eta,\xi\right).
\]
Since $\mathcal{B}_{AJ}\left(f,g\right)\left(\eta,\xi\right)=\mathfrak{T}_{\tilde{I}}\mathcal{B}_{AJ}\left(f,g\right)\left(\xi,\eta\right)$,
we finally have 
\[
\mathcal{F}\mathcal{B}_{A}\left(f,g\right)\left(\xi,\eta\right)=\mathcal{B}_{AJ\tilde{I}}\left(f,g\right)\left(\xi,\eta\right)=\mathcal{B}_{A\mathcal{I}_{2}}\left(f,g\right)\left(\xi,\eta\right).
\]
\end{proof}

\subsection{Additional regularity of submatrices }

Following a known pattern for the Wigner transform, it is interesting
to determine the conditions under which a bilinear time-frequency
distribution can be related to the STFT. 
\begin{definition}
A $d$-block matrix $A=\left(\begin{array}{cc}
A_{11} & A_{12}\\
A_{21} & A_{22}
\end{array}\right)\in\mathbb{R}^{2d\times2d}$ is called left-regular (resp. right-regular) if the submatrices $A_{11},A_{21}\in\mathbb{R}^{d\times d}$
(resp. $A_{12},A_{22}\in\mathbb{R}^{d\times d}$) are invertible. 
\end{definition}
\begin{remark}
It is an easy exercise of linear algebra to prove that $A=\left(\begin{array}{cc}
A_{11} & A_{12}\\
A_{21} & A_{22}
\end{array}\right)\in\GLL$ is left-regular (resp. right-regular) if and only if the matrix $A^{\#}=\left(A^{-1}\right)^{\top}=\left(\begin{array}{cc}
(A^{\#})_{11} & (A^{\#})_{12}\\
(A^{\#})_{21} & (A^{\#})_{22}
\end{array}\right)$ is right-regular (resp. left-regular). Also, beware that $(A^{\#})_{ij}\neq A^{\#}_{ij} = (A_{ij}^{\top})^{-1}$, $i,j=1,2$. 
\end{remark}
\begin{theorem}
\label{right-reg rep}Assume $A\in\GLL$
and right-regular. For every $f,g\in L^{2}\left(\mathbb{R}^{d}\right)$, the following formula holds:
\[
\mathcal{B}_{A}\left(f,g\right)\left(x,\omega\right)=\left|\det A_{12}\right|^{-1}e^{2\pi i A_{12}^{\#}\omega\cdot A_{11}x}V_{\tilde{g}}f\left(c\left(x\right),d\left(\omega\right)\right),\quad x,\omega\in\rd,
\]
where 
\[
c\left(x\right)=\left(A_{11}-A_{12}A_{22}^{-1}A_{21}\right)x,\qquad d\left(\omega\right)=A_{12}^{\#}\omega,\qquad\tilde{g}\left(t\right)=g\left(A_{22}A_{12}^{-1}t\right).
\]
\end{theorem}
\begin{proof}
If $A$ is right-regular, for any $f,g\in L^{2}\left(\mathbb{R}^{d}\right)$
the functions $f'=\mathfrak{T}_{A_{12}}f$ and $g'=\mathfrak{T}_{A_{22}}g$
are well-defined in $L^{2}\left(\mathbb{R}^{d}\right)$. Therefore,
we can write 
\[
f\left(A_{11}x+A_{12}y\right)=T_{-A_{12}^{-1}A_{11}x}f'\left(y\right),\qquad g\left(A_{21}x+A_{22}y\right)=T_{-A_{22}^{-1}A_{21}x}g'\left(y\right)
\]
and thus the integral 
\[
\mathcal{B}_{A}\left(f,g\right)\left(x,\omega\right)=\int_{\mathbb{R}^{d}}e^{-2\pi i\omega y}f\left(A_{11}x+A_{12}y\right)\overline{g\left(A_{21}x+A_{22}y\right)}dy
\]
is defined pointwise. Introducing the change of variable $z=A_{11}x+A_{12}y$
gives the claimed representation. 
\end{proof}
Let us exhibit the continuity properties of bilinear time-frequency distributions on Lebesgue spaces. 
\begin{proposition}
\label{right-regular continuity}Assume $A\in\GLL$
and right-regular. For any $1<p<\infty$ and $q\ge2$ such that $q'\le p\le q$, $f\in L^{p}\left(\mathbb{R}^{d}\right)$
and $g\in L^{p'}\left(\mathbb{R}^{d}\right)$, we have
\begin{enumerate}[label=(\roman*)]
\item $\mathcal{B}_{A}(f,g)\in L^{q}(\rdd)$,
with
\begin{equation}\label{qnorm estimate}
\left\Vert \mathcal{B}_{A}(f,g)\right\Vert _{q}\le\frac{\left\Vert f\right\Vert _{p}\left\Vert g\right\Vert _{p'}}{\left|\det A\right|^{\frac{1}{q}}\left|\det A_{12}\right|^{\frac{1}{p}-\frac{1}{q}}\left|\det A_{22}\right|^{\frac{1}{p'}-\frac{1}{q}}}.
\end{equation}
\item $\mathcal{B}_{A}(f,g)\in C_{0}(\rdd)$. In particular, $\mathcal{B}_{A}(f,g)\in L^{\infty}(\rdd)$.
\end{enumerate}
\end{proposition}
\begin{proof}
~
\begin{enumerate}[label=(\roman*)]
\item We use the result \cite[Proposition 3.1]{bdo lebesgue} for the $L^{q}$-norm
of the STFT, that is:
\begin{flalign*}
\left\Vert \mathcal{B}_{A}\left(f,g\right)\right\Vert _{q} & =\left|\det A_{12}\right|^{-1}\left\Vert V_{\tilde{g}}f\left(c\left(x\right),d\left(\omega\right)\right)\right\Vert _{q}\\
 & =\left|\det A_{12}\right|^{-1}\left\Vert D_{W}V_{\tilde{g}}f\right\Vert _{q}\\
 & =\left|\det A_{12}\right|^{-1}\left|\det W\right|^{-1/q}\left\Vert V_{\tilde{g}}f\right\Vert _{q}\\
 & \le\left|\det A_{12}\right|^{-1}\left|\det W\right|^{-1/q}\left\Vert f\right\Vert _{p}\left\Vert \tilde{g}\right\Vert _{p'}\\
 & =\left|\det A_{12}\right|^{-1}\left|\det W\right|^{-1/q}\left\Vert f\right\Vert _{p}\left(\frac{\left|\det A_{12}\right|^{1/p'}}{\left|\det A_{22}\right|^{1/p'}}\left\Vert g\right\Vert _{p'}\right)\\
 & =\frac{\left\Vert f\right\Vert _{p}\left\Vert g\right\Vert _{p'}}{\left|\det W\right|^{1/q}\left|\det A_{12}\right|^{1/p}\left|\det A_{22}\right|^{1/p'}}.
\end{flalign*}
where $D_{M}$ denotes the dilation by the invertible matrix $M\in\GLL$,
namely $D_{M}F\left(z\right)=F\left(Mz\right)$, $z\in\rdd$.
In particular, since 
\[
W=\left(\begin{array}{cc}
A_{11}-A_{12}A_{22}^{-1}A_{21} & 0\\
0 & A_{12}^{\#}
\end{array}\right),
\]
we see that 
\[
\det W=\det\left(A_{11}-A_{12}A_{22}^{-1}A_{21}\right)\cdot\frac{1}{\det A_{12}}=\frac{\det A}{\det A_{12}\det A_{22}}\ne0,
\]
hence
\[
\left\Vert \mathcal{B}_{A}\left(f,g\right)\right\Vert _{q}\le\frac{\left\Vert f\right\Vert _{p}\left\Vert g\right\Vert _{p'}}{\left|\det A\right|^{1/q}\left|\det A_{12}\right|^{\frac{1}{p}-\frac{1}{q}}\left|\det A_{22}\right|^{\frac{1}{p'}-\frac{1}{q}}}.
\]
\item Arguing by density, there exist sequences $\left\{ f_{n}\right\} ,\left\{ g_{n}\right\} \in\mathcal{S}\left(\mathbb{R}^{d}\right)$
such that $f_{n}\rightarrow f$ in $L^{p}$ and $g_{n}\rightarrow g$
in $L^{p'}$. Since $\cB_{A}\left(f_{n},g_{n}\right)\in\mathcal{S}\left(\mathbb{R}^{2d}\right)\subset C_{0}\left(\mathbb{R}^{2d}\right)$ by Proposition \ref{def bilA triple},
we have 
\begin{alignat*}{1}
\left\Vert \cB_{A}\left(f_{n},g_{n}\right)-\cB_{A}\left(f,g\right) \right\Vert _{\infty} & =\left\Vert \cB_{A}\left(f_n,g_n\right)-\cB_{A}(f_n,g)+\cB_{A}(f_n,g)-\cB_{A}(f,g)\right\Vert _{\infty}\\
& \le\left\Vert \cB_{A}(f,g_n-g)\right\Vert _{\infty}+\left\Vert \cB_{A}(f-f_n,g)\right\Vert _{\infty}\\
& \le\frac{\left\Vert f\right\Vert _{p}\left\Vert g_n-g\right\Vert _{p'}+\left\Vert f-f_n\right\Vert _{p}\left\Vert g\right\Vert _{p'}}{\left|\det A\right|^{\frac{1}{q}}\left|\det A_{12}\right|^{\frac{1}{p}-\frac{1}{q}}\left|\det A_{22}\right|^{\frac{1}{p'}-\frac{1}{q}}}.\\
\end{alignat*}
Since the sequence $\left\{ \left\Vert f_{n}\right\Vert _{p}\right\} $
is bounded, we then have $$\lim_{n\rightarrow\infty}\left\Vert\cB_{A}\left(f_{n},g_{n}\right)-\cB_{A}\left(f,g\right)\right\Vert _{\infty}=0.$$
This implies $\cB_{A}(f,g)\in C_{0}\left(\mathbb{R}^{2d}\right)$, as desired.
\end{enumerate}
\end{proof}
\begin{corollary}[Riemann-Lebesgue for the STFT]
Let $1<p<\infty$, $f\in L^{p}\left(\mathbb{R}^{d}\right)$ and $g\in L^{p'}\left(\mathbb{R}^{d}\right)$. Then, $V_{g}f\in C_{0}\left(\rdd\right)$
\end{corollary}

\begin{proof}
	It follows from Proposition \ref{qnorm estimate}, since $V_g(f)=\cB_{A_{ST}}(f,g)$ with $A_{ST}$ right-regular, cf. \eqref{ASTFT}.
\end{proof}
We conclude this section by mentioning that right-regularity is indeed
a necessary condition for the continuity of $\mathcal{B}_{A}\left(f,g\right)$,
as proved in the following result. 
\begin{theorem}[{\cite[Theorem 1.2.9]{bayer}}]
 Assume $A\in\GLL$ such that $\det A_{22}\ne0$
but $\det A_{12}=0$. Then, there exist $f,g\in L^{2}\left(\mathbb{R}^{d}\right)$
such that $\mathcal{B}_{A}\left(f,g\right)$ is not a continuous function
on $\rdd$. 
\end{theorem}

\subsection{Orthogonality and inversion formulas}

A fundamental and desirable property for a time-frequency
distribution is the validity of the so-called orthogonality relations. These are the analogue of Parseval's Theorem for the Fourier transform and are also known as
Moyal's formula for the Wigner distribution. From the orthogonality relations one can also derive an inversion formula
allowing to recover the original signal from the knowledge of its
time-frequency representation. The connection between
these two issues is clarified by the following abstract result.
\begin{theorem}
\label{orthogonality -> inversion}Let $H_{1},H_{2}$ be complex
Hilbert spaces and assume that the members of the family of linear
bounded operators $\left\{ T_{g}:H_{1}\rightarrow H_{2}\,|\,g\in H_{1}\right\} $
satisfy an orthogonality relation of the following type: for any fixed
$g,\gamma\in H_{1}$ there exists $C_{g,\gamma}\in\mathbb{C}$ such
that 
\[
\left\langle T_{g}f,T_{\gamma}h\right\rangle _{H_{2}}=C_{g,\gamma}\left\langle f,h\right\rangle _{H_{1}}\qquad\forall f,h\in H_{1}.
\]

If $C_{g,\gamma}\ne0$, the following inversion formula holds:
\[
f=\frac{1}{C_{g,\gamma}}T_{\gamma}^{*}T_{g}f,\qquad\forall f\in H_{1}.
\]
\end{theorem}
\begin{proof}
For any $f,h\in H_{1}$ we have 
\[
\left\langle T_{\gamma}^{*}T_{g}f,h\right\rangle _{H_{1}}=\left\langle T_{g}f,T_{\gamma}h\right\rangle _{H_{2}}=C_{g,\gamma}\left\langle f,h\right\rangle _{H_{1}},
\]
hence the claimed formula.

\end{proof}
\begin{remark}
For the sake of completeness, we remark that a similar pathway can
be traced under slightly weaker assumptions, namely for any linear bounded
 operator $T:H_{1}\rightarrow H_{2}$ which is a non-trivial
constant multiple of an isometry: 
\[
\exists C>0\,|\,\left\Vert Tf\right\Vert _{H_{2}}=C\left\Vert f\right\Vert _{H_{1}}.
\]
Indeed, by polarization identity, for any $f,h\in H_{1}$ we have
\begin{flalign*}
\left\langle T^{*}Tf,h\right\rangle _{H_{1}} & =\left\langle Tf,Th\right\rangle _{H_{2}}  =\frac{1}{4}\sum_{\substack{z\in\mathbb{C}\\
z^{4}=1
}
}z\left\Vert Tf+zTh\right\Vert _{H_{2}}^{2}\\
 & =\frac{1}{4}\sum_{\substack{z\in\mathbb{C}\\
z^{4}=1
}
}zC^{2}\left\Vert f+zh\right\Vert _{H_{1}}^{2} =C^{2}\left\langle f,h\right\rangle _{H_{1}}.
\end{flalign*}
Hence, $f=\displaystyle\frac{1}{C^{2}}T^{*}Tf.$
\end{remark}
We then generalize Moyal's formula to MWDs. 
\begin{theorem}[Orthogonality relations]
 Let $A\in\GLL$ and $f_{1},f_{2},g_{1},g_{2}\in L^{2}\left(\mathbb{R}^{d}\right)$.
Then 
\begin{equation}\label{ortrel}
\left\langle \mathcal{B}_{A}\left(f_{1},g_{1}\right),\mathcal{B}_{A}\left(f_{2},g_{2}\right)\right\rangle _{L^{2}\left(\rdd\right)}=\frac{1}{\left|\det A\right|}\left\langle f_{1},f_{2}\right\rangle _{L^{2}\left(\mathbb{R}^{d}\right)}\overline{\left\langle g_{1},g_{2}\right\rangle _{L^{2}\left(\mathbb{R}^{d}\right)}}.
\end{equation}
In particular, 
\[
\left\Vert \mathcal{B}_{A}\left(f,g\right)\right\Vert _{L^{2}\left(\rdd\right)}=\frac{1}{\left|\det A\right|^{1/2}}\left\Vert f\right\Vert _{L^{2}\left(\mathbb{R}^{d}\right)}\left\Vert g\right\Vert _{L^{2}\left(\mathbb{R}^{d}\right)}.
\]
Thus, the representation  $\mathcal{B}_{A,g}:L^{2}\left(\mathbb{R}^{d}\right)\ni f\mapsto\mathcal{B}_{A}\left(f,g\right)\in L^{2}\left(\rdd\right)$
is a non-trivial constant multiple of an isometry whenever $g\not\equiv0$. 
\end{theorem}
\begin{proof}
Since $\mathcal{F}_{2}$ is a unitary operator on $L^{2}\left(\mathbb{R}^{d}\right)$
and $\mathfrak{T}_{A}$ is unitary up to the constant factor $\left|\det A\right|^{-1/2}$,
we have 
\begin{flalign*}
\left\langle \mathcal{B}_{A}\left(f_{1},g_{1}\right),\mathcal{B}_{A}\left(f_{2},g_{2}\right)\right\rangle  & =\left\langle \mathcal{F}_{2}\mathfrak{T}_{A}\left(f_{1}\otimes\overline{g_{1}}\right),\mathcal{F}_{2}\mathfrak{T}_{A}\left(f_{2}\otimes\overline{g_{2}}\right)\right\rangle \\
 & =\frac{1}{\left|\det A\right|}\left\langle f_{1}\otimes\overline{g_{1}},f_{2}\otimes\overline{g_{2}}\right\rangle \\
 & =\frac{1}{\left|\det A\right|}\left\langle f_{1},f_{2}\right\rangle \overline{\left\langle g_{1},g_{2}\right\rangle }.
\end{flalign*}
\end{proof}
\begin{corollary}
If $\left(e_{n}\right)_{n\in\mathbb{N}}$ is an orthonormal basis
for $L^{2}\left(\mathbb{R}^{d}\right)$, then \[\left\{ \left|\det A\right|^{1/2}\mathcal{B}_{A}\left(e_{m},e_{n}\right)\,|\,m,n\in\mathbb{N}\right\}\]
is an orthonormal basis for $L^{2}\left(\rdd\right)$.
\end{corollary}
\begin{proof}
From orthogonality relation we have
\[
\left\langle \left|\det A\right|^{1/2}\mathcal{B}_{A}\left(e_{m},e_{n}\right),\left|\det A\right|^{1/2}\mathcal{B}_{A}\left(e_{i},e_{j}\right)\right\rangle =\left\langle e_{m},e_{i}\right\rangle \overline{\left\langle e_{n},e_{j}\right\rangle }=\delta_{m,i}\delta_{n,j}.
\]
This proves that $\left\{ \left|\det A\right|^{1/2}\mathcal{B}_{A}\left(e_{m},e_{n}\right)\,|\,m,n\in\mathbb{N}\right\} $
is an orthonormal family in $L^{2}\left(\mathbb{R}^{d}\right)$, its
span being a complete subset of $L^{2}\left(\mathbb{R}^{d}\right)$,
hence the thesis. 
\end{proof}
Before establishing an inversion formula, it is convenient to explicitly
characterize the adjoint of $\mathcal{B}_{A}\left(f,g\right)$. 
\begin{proposition}
Let $A\in\GLL$ and fix $g\in L^{2}\left(\mathbb{R}^{d}\right)$.
Then, 
\[
\mathcal{B}_{A,g}^{*}:L^{2}\left(\rdd\right)\rightarrow L^{2}\left(\mathbb{R}^{d}\right),\qquad\mathcal{B}_{A,g}^{*}H\left(x\right)=\frac{1}{\left|\det A\right|}\int_{\mathbb{R}^{d}}\mathfrak{T}_{A^{\star}}\mathcal{F}_{2}H\left(x,y\right)g\left(y\right)dy,
\]
where 
\[
A^{\star}=A^{-1}\mathcal{I}_{2}\in\GLL.
\]
\end{proposition}
\begin{proof}
Set for convenience 
\[
h\left(x\right)=\frac{1}{\left|\det A\right|}\int_{\mathbb{R}^{d}}\mathfrak{T}_{A^{\star}}\mathcal{F}_{2}H\left(x,y\right)g\left(y\right)dy,
\]
and notice that if $H\in L^{2}\left(\rdd\right)$ then
$h\in L^{2}\left(\mathbb{R}^{d}\right)$. Let $f\in L^{2}\left(\mathbb{R}^{d}\right)$
and $H\in L^{2}\left(\mathbb{R}^{d}\right)$, then

\begin{flalign*}
\left\langle \mathcal{B}_{A,g}f,H\right\rangle  & =\left\langle \mathcal{F}_{2}\mathfrak{T}_{A}\left(f\otimes\overline{g}\right),H\right\rangle \\
 & =\left\langle \mathfrak{T}_{A}\left(f\otimes\overline{g}\right),\mathcal{F}_{2}^{*}H\right\rangle \\
 & =\left\langle \mathfrak{T}_{A}\left(f\otimes\overline{g}\right),\mathfrak{T}_{\mathcal{I}_{2}}\mathcal{F}_{2}H\right\rangle \\
 & =\left\langle f\otimes\overline{g},\left|\det A\right|^{-1}\mathfrak{T}_{A^{-1}}\mathfrak{T}_{\mathcal{I}_{2}}\mathcal{F}_{2}H\right\rangle \\
 & =\left\langle f\otimes\overline{g},\left|\det A\right|^{-1}\mathfrak{T}_{A^{\star}}\mathcal{F}_{2}H\right\rangle \\
 & =\left\langle f,h\right\rangle ,
\end{flalign*}
where the last equality follows from Fubini's theorem. 
\end{proof}
\begin{corollary}[Inversion formula for bilinear TF representations]

Assume $A\in\GLL$ and fix $g,\gamma\in L^{2}\left(\mathbb{R}^{d}\right)$
such that $\left\langle g,\gamma\right\rangle \ne0$. Then, for any
$f\in L^{2}\left(\mathbb{R}^{d}\right)$, the following inversion
formula holds:
\[
f=\frac{\left|\det A\right|}{\overline{\left\langle g,\gamma\right\rangle }}\mathcal{B}_{A,\gamma}^{*}\mathcal{B}_{A,g}f.
\]
\end{corollary}
\begin{proof}
It is an immediate consequence of the general result in Theorem \ref{orthogonality -> inversion}
with $T_{g}=\mathcal{B}_{A,g}$ and $C_{g,\gamma}=\left|\det A\right|^{-1}\overline{\left\langle g,\gamma\right\rangle }$. 
\end{proof}
Under more restrictive assumptions, a pointwise inversion formula
can be provided without resorting to the adjoint operator. First,
notice that $\mathcal{B}_{A}f$ determines $f$ only up to a phase
factor: whenever $c\in\mathbb{C}$, $\left|c\right|=1$, we have 
\[
\mathcal{B}_{A}\left(cf\right)=\left|c\right|^{2}\mathcal{B}_{A}\left(f\right)=\mathcal{B}_{A}\left(f\right).
\]

\begin{theorem}[Pointwise inversion formula]

Assume $A\in\GLL$ and set 
\[
A^{-1}=\left(\begin{array}{cc}
(A^{-1})_{11} & (A^{-1})_{12}\\
(A^{-1})_{21} & (A^{-1})_{22}
\end{array}\right).
\]
For any $f\in\mathcal{S}\left(\mathbb{R}^{d}\right)$ such that $f\left(0\right)\ne0$,
we have
\[
f\left(x\right)=\frac{1}{\overline{f\left(0\right)}}\int_{\mathbb{R}^{d}}e^{2\pi i(A^{-1})_{21}x\cdot\omega}\mathcal{B}_{A}f\left((A^{-1})_{11}x,\omega\right)d\omega.
\]
All other solutions have the form $cf$, where $c\in\mathbb{C}$,
$\left|c\right|=1$. 
\end{theorem}
\begin{proof}
By inverting the operators $\mathcal{F}_{2}$ and $\mathfrak{T}_{A}$,
we have

\begin{align*}
f\left(x\right)\overline{f\left(y\right)} & =\left(\mathfrak{T}_{A^{-1}}\mathcal{F}_{2}^{-1}\mathcal{B}_{A}f\right)\left(x,y\right) \\
& =\int_{\mathbb{R}^{d}}e^{2\pi i\left((A^{-1})_{21}x+(A^{-1})_{22}y\right)\omega}\mathcal{B}_{A}f\left((A^{-1})_{11}x+(A^{-1})_{12}y,\omega\right)d\omega.
\end{align*}
Setting $y=0$ gives the desired formula.
\end{proof}
We conclude this section by providing an inversion formula for representations
associated with right-regular matrices. The easy proof is left to
the interested reader. 
\begin{proposition}
Let $A\in\GLL$ be a right-regular
matrix, and $g,\gamma\in L^{2}\left(\mathbb{R}^{d}\right)$ such that
$\left\langle g,\gamma\right\rangle \ne0$. The following inversion
formula (to be interpreted as vector-valued integral in $L^{2}\left(\mathbb{R}^{d}\right)$)
holds for any $f\in L^{2}\left(\mathbb{R}^{d}\right)$:
\[
f=\frac{1}{\left\langle g,\gamma\right\rangle }\int_{\rdd}\mathcal{B}_{A,\gamma}f\left(x,\omega\right)\frac{e^{-2\pi iA_{12}^{\#}\omega\cdot A_{11}x}}{\left|\det A_{12}\right|}M_{d\left(\omega\right)}T_{c\left(x\right)}\tilde{g}dxd\omega,
\]
where
\[
c\left(x\right)=\left(A_{11}-A_{12}A_{22}^{-1}A_{21}\right)x,\qquad d\left(\omega\right)=A_{12}^{\#}\omega,\qquad\tilde{g}\left(t\right)=g\left(A_{22}A_{12}^{-1}t\right).
\]
\end{proposition}

\subsection{Covariance and short-time product formulas}

A key property for a time-frequency distribution is its behaviour
under the action of time-frequency shifts. We prove a covariance formula
for MWDs. 
\begin{theorem}[Covariance formula]
\label{cov formula}Let $A\in\GLL$.
For any $f,g\in L^{2}\left(\mathbb{R}^{d}\right)$ and $a,b,\alpha,\beta\in\mathbb{R}^{d}$,
the following formula holds:
\begin{align}\label{covform}
\mathcal{B}_{A}\left(M_{\alpha}T_{a}f,M_{\beta}T_{b}g\right)\left(x,\omega\right)& =e^{2\pi i\sigma s}M_{\left(\rho,-s\right)}T_{\left(r,\sigma\right)}\mathcal{B}_{A}\left(f,g\right)\left(x,\omega\right) \\
& =e^{2\pi i\sigma s}e^{2\pi i\left(x\rho-\omega s\right)}\mathcal{B}_{A}\left(f,g\right)\left(x-r,\omega-\sigma\right),
\end{align}

where 
\[
\left(\begin{array}{c}
r\\
s
\end{array}\right)=A^{-1}\left(\begin{array}{c}
a\\
b
\end{array}\right),\qquad\left(\begin{array}{c}
\rho\\
\sigma
\end{array}\right)=A^{\top}\left(\begin{array}{c}
\alpha\\
-\beta
\end{array}\right).
\]
\end{theorem}
\begin{proof}
We employ the results in Lemmas \ref{tfs coortr} and \ref{tfs partial Fou},
and the notation introduced in the claim:
\begin{alignat*}{1}
\mathcal{B}_{A}\left(M_{\alpha}T_{a}f,M_{\beta}T_{b}g\right) & =\mathcal{F}_{2}\mathfrak{T}_{A}\left(M_{\left(\alpha,-\beta\right)}T_{\left(a,b\right)}\left(f\otimes\overline{g}\right)\right)\\
 & =\mathcal{F}_{2}\left(M_{\left(\rho,\sigma\right)}T_{\left(r,s\right)}\mathfrak{T}_{A}\left(f\otimes\overline{g}\right)\right)\\
 & =e^{2\pi i\sigma s}M_{\left(\rho,-s\right)}T_{\left(r,\sigma\right)}\mathcal{F}_{2}\mathfrak{T}_{A}\left(f\otimes\overline{g}\right)\\
 & =e^{2\pi i\sigma s}M_{\left(\rho,-s\right)}T_{\left(r,\sigma\right)}\mathcal{B}_{A}\left(f,g\right).
\end{alignat*}
This concludes the proof.
\end{proof}
\begin{corollary}
\label{cov form inv}Let $A\in\GLL$.
For any $f,g\in L^{2}\left(\mathbb{R}^{d}\right)$ and $a,b,\alpha,\beta\in\mathbb{R}^{d}$,
the following formula holds:
\begin{equation}\label{covforminv}
M_{\left(\alpha,\beta\right)}T_{\left(a,b\right)}\mathcal{B}_{A}\left(f,g\right)\left(x,\omega\right)=e^{2\pi ib\beta}\mathcal{B}_{A}\left(M_{\rho}T_{r}f,M_{\sigma}T_{s}g\right)\left(x,\omega\right),\quad \phas\in\rdd,
\end{equation}
where 
\[
\left(\begin{array}{c}
r\\
s
\end{array}\right)=A\mathcal{I}_{2}\left(\begin{array}{c}
a\\
\beta
\end{array}\right),\qquad\left(\begin{array}{c}
\rho\\
\sigma
\end{array}\right)=\mathcal{I}_{2}A^{\#}\left(\begin{array}{c}
\alpha\\
b
\end{array}\right).
\]
\end{corollary}

Notice that we recapture the covariance formula for the $\tau$-Wigner distribution with $A=A_{\tau}$ as in \eqref{Atau}, cf. \cite[Prop. 3.3]{cnt18}. In particular, for $\tau=1/2$, $\alpha=\beta$ and $a=b$, the covariance formula for the Wigner distribution follows:
\begin{equation}\label{cov wig}
W(M_{\alpha}T_a f,M_{\alpha}T_a g)(x,\omega)=W(f,g)(x-a,\omega-\alpha).
\end{equation}
Furthermore, the covariance properties established in Theorem \ref{cov formula} easily extend to any modulation space $M^p(\rd)$,  for every  $1\leq p\leq \infty$.

We now establish an amazing representation result for the STFT of
a bilinear time-frequency distribution. This will allow to enlarge
the functional framework to modulation and Wiener amalgam spaces with
minimum effort. 
\begin{theorem}[Short-time product formula]
\label{magic formula}Assume $A\in\GLL$
and $f,g,\psi,\phi\in L^{2}\left(\mathbb{R}^{d}\right)$, and set
$z=\left(z_{1},z_{2}\right)$, $\zeta=\left(\zeta_{1},\zeta_{2}\right)\in\rdd$.
Then, 
\begin{equation}\label{STP}
V_{\mathcal{B}_{A}\left(\phi,\psi\right)}\mathcal{B}_{A}\left(f,g\right)\left(z,\zeta\right)=e^{-2\pi iz_{2}\zeta_{2}}V_{\phi}f\left(a,\alpha\right)\overline{V_{\psi}g\left(b,\beta\right)},
\end{equation}
where 
\[
\left(\begin{array}{c}
a\\
b
\end{array}\right)=A\mathcal{I}_{2}\left(\begin{array}{c}
z_{1}\\
\zeta_{2}
\end{array}\right)=\left(\begin{array}{c}
A_{11}z_{1}-A_{12}\zeta_{2}\\
A_{21}z_{1}-A_{22}\zeta_{2}
\end{array}\right),
\]
\[
\left(\begin{array}{c}
\alpha\\
\beta
\end{array}\right)=\mathcal{I}_{2}A^{\#}\left(\begin{array}{c}
\zeta_{1}\\
z_{2}
\end{array}\right)=\left(\begin{array}{c}
(A^{\#})_{11}\zeta_{1}+(A^{\#})_{12}z_{2}\\
-(A^{\#})_{21}\zeta_{1}-(A^{\#})_{22}z_{2}
\end{array}\right).
\]
\end{theorem}
\begin{proof}
It is a matter of computation:

\begin{alignat*}{1}
V_{\mathcal{B}_{A}\left(\phi,\psi\right)}\mathcal{B}_{A}\left(f,g\right)\left(z,\zeta\right) & =\left\langle \mathcal{B}_{A}\left(f,g\right),M_{\zeta}T_{z}\mathcal{B}_{A}\left(\phi,\psi\right)\right\rangle \\
 & =e^{-2\pi iz_{2}\zeta_{2}}\left\langle \mathcal{B}_{A}\left(f,g\right),\mathcal{B}_{A}\left(M_{\alpha}T_{a}\phi,M_{\beta}T_{b}\psi\right)\right\rangle \\
 & =e^{-2\pi iz_{2}\zeta_{2}}\left\langle f,M_{\alpha}T_{a}\phi\right\rangle \overline{\left\langle g,M_{\beta}T_{b}\psi\right\rangle }\\
 & =e^{-2\pi iz_{2}\zeta_{2}}V_{\phi}f\left(a,\alpha\right)\overline{V_{\psi}g\left(b,\beta\right)},
\end{alignat*}
where we used the orthogonality relations and Corollary \ref{cov form inv},
with
\[
\left(\begin{array}{c}
a\\
b
\end{array}\right)=A\left(\begin{array}{c}
z_{1}\\
-\zeta_{2}
\end{array}\right),\qquad\left(\begin{array}{c}
\alpha\\
\beta
\end{array}\right)=\mathcal{I}_{2}A^{\#}\left(\begin{array}{c}
\zeta_{1}\\
z_{2}
\end{array}\right).
\]
\end{proof}

\section{Cohen's class and perturbations}
This section is the core of our study. We shall prove Theorem \ref{maint}.
Recall first  the definition of Cohen's class, a family of phase-space representations
obtained by convolving the Wigner transform with a tempered distribution, as detailed below. 
\begin{definition}[\cite{Grochenig_2001_Foundations}]
A time-frequency distribution $Q$ belongs to the Cohen's class if
there exists a tempered distribution $\theta\in\mathcal{S}'\left(\rdd\right)$
such that
\[
Q\left(f,g\right)=W\left(f,g\right)*\theta,\qquad\forall f,g\in\mathcal{S}\left(\mathbb{R}^{d}\right).
\]
\end{definition}
The MWDs belonging to Cohen's class can be completely
characterized, as detailed in Theorem \ref{maint}, that we are going to prove.

\begin{proof}[Proof of Theorem \ref{maint}]
Let us first prove necessity. Observe that a member of Cohen's class necessarily satisfies the covariance property \eqref{cov wig}: 
\[
Q\left(M_{\omega}T_{x}f\right)=T_{\left(x,\omega\right)}Qf, \quad \forall f\in\mathcal{S}(\mathbb{R}^{d}).
\]

By Theorem \ref{cov formula}, with $\alpha=\beta=\omega$, $a=b=x$
and $f=g$, we get
\[
\left(\rho,\sigma\right)=\left(0,\omega\right),\qquad\left(r,s\right)=\left(x,0\right).
\]
Converting these into conditions for the matrix $A=\left(\begin{array}{cc}
A_{11} & A_{12}\\
A_{21} & A_{22}
\end{array}\right)$ yields
\[
\rho=0\Rightarrow\left(A_{11}^{\top}-A_{21}^{\top}\right)\omega=0\Rightarrow\quad A_{11}=A_{21}=N,\qquad N\in\mathbb{R}^{d\times d}.
\]
\[
\sigma=\omega\Rightarrow\left(A_{12}^{\top}-A_{22}^{\top}\right)\omega=\omega\Rightarrow\quad A_{12}-A_{22}=I.
\]
Setting $A_{22}=M-(1/2)I$ for some $M\in\mathbb{R}^{d\times d}$
(other parametrizations are of course allowed), the block structure
of $A$ is thus determined by
\[
A=\left(\begin{array}{cc}
N & M+(1/2)I\\
N & M-(1/2)I
\end{array}\right).
\]
In order to exploit the conditions on $\left(r,s\right)$, notice
that $A$ is assumed to be invertible. From \cite[App. A - Lemma. 4]{folland}
we in fact have
\[
\det A=\left(-1\right)^{d}\det N\ne0,
\]
hence $N\in\mathrm{GL}\left(d,\mathbb{R}\right)$. With this additional
information we are able to explicitly compute $A^{-1}$, namely
\[
A^{-1}=\left(\begin{array}{cc}
-N^{-1}\left(M-(1/2)I\right) & N^{-1}\left(M+(1/2)I\right)\\
I & -I
\end{array}\right).
\]
Then
\[
\left(\begin{array}{c}
r\\
s
\end{array}\right)=A^{-1}\left(\begin{array}{c}
x\\
x
\end{array}\right)=\left(\begin{array}{c}
N^{-1}x\\
0
\end{array}\right),
\]
hence $s=0$ is automatically fulfilled and we get 
\[
r=x\Rightarrow N^{-1}x=x\Rightarrow\quad N=I.
\]
In conclusion, if $\mathcal{B}_{A}$ belongs to the Cohen's class,
then $A$ has the form \eqref{AM}. 

For what concerns sufficiency, assume that $A=A_{M}$ has this prescribed
form. We shall show that $\mathcal{B}_{A_{M}}=W*\theta_{M}$ for some $\theta_{M}\in\mathcal{S}'\left(\rdd\right)$.
Applying the symplectic Fourier transform to both sides, this is equivalent
to showing that, for any $f,g\in\mathcal{S}\left(\mathbb{R}^{d}\right)$,
\begin{equation}\label{eq:cohen symp fou}
\mathcal{F}_{\sigma}\mathcal{B}_{A_{M}}\left(f,g\right)=\mathcal{F}_{\sigma}W\left(f,g\right)\cdot\mathcal{F}_{\sigma}\theta_{M}=Amb\left(f,g\right)\cdot\mathcal{F}_{\sigma}\theta_{M},
\end{equation}
where $Amb(f,g)$ is defined in \eqref{ambiguity}. 
From \eqref{FouB}, for any $\xi,\eta\in\mathbb{R}^{d}$ we have

\[
\mathcal{F}\mathcal{B}_{A_{M}}\left(f,g\right)\left(\xi,\eta\right)=\mathcal{B}_{A_{M}J}\left(f,g\right)\left(\eta,\xi\right),
\]
where 
\[
A_{M}J=\left(\begin{array}{cc}
-\left(M+(1/2)I\right) & I\\
-\left(M-(1/2)I\right) & I
\end{array}\right),
\]
therefore
\begin{align*}
\mathcal{F}_{\sigma}\mathcal{B}_{A_{M}}\left(f,g\right)\left(\xi,\eta\right) 
& = \mathcal{F}\mathcal{B}_{A_{M}}\left(f,g\right)\left(J(\xi,\eta)\right) 
\\ & = \mathcal{F}\mathcal{B}_{A_{M}}\left(f,g\right)\left(\eta,-\xi)\right)
\\ & =\mathcal{B}_{A_{M}J}\left(f,g\right)\left(-\xi,\eta\right)\\
 & =\int_{\mathbb{R}^{d}}e^{-2\pi i\eta t}f\left(t+\left(M+\frac{1}{2}I\right)\xi\right)\overline{g\left(t+\left(M-\frac{1}{2}I\right)\xi\right)}dt.
\end{align*}
The substitution $t+\left(M-(1/2)I\right)\xi=z-{\xi}/{2}$
yields
\begin{alignat*}{1}
\mathcal{F}_{\sigma}\mathcal{B}_{A_{M}}\left(f,g\right)\left(\xi,\eta\right) 
 & =e^{2\pi i\eta\cdot M\xi}\int_{\mathbb{R}^{d}}e^{-2\pi i\eta z}f\left(z+\frac{\xi}{2}\right)\overline{g\left(z-\frac{\xi}{2}\right)}dz\\
 & =e^{2\pi i\eta\cdot M\xi}\cdot Amb\left(f,g\right)\left(\xi,\eta\right),
\end{alignat*}
so that
\[
\mathcal{F}_{\sigma}\theta_{M}\left(\xi,\eta\right)=e^{2\pi i\eta\cdot M\xi}.
\]
Defining $\chi_{M}\left(\xi,\eta\right)\coloneqq e^{2\pi i\eta\cdot M\xi}$ and using $\cF_{\sigma}^2=I$ (the symplectic Fourier transform is an involution), we finally obtain 
\[
\theta_{M}\left(\xi,\eta\right)=\mathcal{F}_{\sigma}\chi_{M}\left(\xi,\eta\right)=\mathcal{F}_{\sigma}\left[e^{2\pi i\eta\cdot M\xi}\right]\in\mathcal{S}'\left(\rdd\right).
\]
\end{proof}
To summarise, Theorem \ref{maint} and \eqref{BAe} yield the following family of time-frequency representations belonging to the Cohen's class:
\begin{equation}\label{BAM}
\cB_{A_M}f(x,\omega)=\int_{\rd}e^{-2\pi i\omega y}f\left(x+\left(M+\frac{1}{2}I\right)y\right)\overline{f\left(x+\left(M-\frac{1}{2}I\right) y\right)}dy.
\end{equation}

If one wants to underline the particular symmetry with respect to
the Wigner distribution $\left(M=0\right)$, the following point of
view on MWDs in the Cohen's class can be assumed: we feel that this
class of distributions is, in some heuristic sense, a family of ``linear
perturbations'' of the Wigner distribution. This interpretation can
be justified both at the level of matrices and at the level of kernels,
but the main insight here is that we are considering a simple family
of time-frequency distributions in the Cohen's class and our aim is
to enlighten which properties of Wigner or $\tau$-Wigner distributions
is ``stable'' under this type of perturbation. 

 In order to enforce this viewpoint, let us highlight the effect of the perturbation on the Gaussian signal. 
\begin{lemma}[Perturbed representation of a Gaussian signal]\label{BAf}
Consider  $A=A_M\in \GLL$ as in \eqref{AM} and $\varphi_{\lambda}\left(t\right)=e^{-\pi t^{2}/\lambda}$, $\lambda>0$. Then,
\begin{multline}\label{BAgauss}
\mathcal{B}_{A}\varphi_{\lambda}\left(x,\omega\right)=\left(2\lambda\right)^{d/2}\det\left(S\right)^{-1/2}e^{-2\pi x^{2}/\lambda}\\ \cdot e^{8\pi\left(M^{\top}x\cdot S^{-1}M^{\top}x\right)/\lambda}e^{8\pi iS^{-1}\omega\cdot M^{\top}x}e^{-2\pi\lambda\omega S^{-1}\omega},
\end{multline}
where $S=I+4M^{\top}M\in\mathbb{R}^{d\times d}$.
\end{lemma}
\begin{proof} Using the definition of $\mathcal{B}_{A}\varphi_{\lambda}$, we can write
\begin{alignat*}{1}
\mathcal{B}_{A}\varphi_{\lambda}\left(x,\omega\right) & =\int_{\mathbb{R}^{d}}e^{-2\pi i\omega y}\exp\left[-\frac{\pi}{\lambda}\left(\left(x+\left(M+\frac{1}{2}I\right)y\right)^{2}+\left(x+\left(M-\frac{1}{2}I\right)y\right)^{2}\right)\right]dy\\
& =\int_{\mathbb{R}^{d}}e^{-2\pi i\omega y}\exp\left[-\frac{\pi}{\lambda}\left(2x^{2}+4xMy+2\left(My\right)^{2}+\frac{1}{2}y^{2}\right)\right]dy\\
& =\int_{\mathbb{R}^{d}}e^{-2\pi i\omega y}\exp\left[-\frac{\pi}{2\lambda}\left(\left(I+4M^{\top}M\right)y\cdot y+2\left(4M^{\top}x\right)y+4x^{2}\right)\right]dy.
\end{alignat*}
Now set $S=I+4M^{\top}M\in\mathbb{R}^{d\times d}$, $t=4M^{\top}x\in\mathbb{R}^{d}$
and $r=4x^{2}$. Notice that, in particular, $S$ is a symmetric  positive-definite matrix, hence invertible. Therefore, 
\begin{alignat*}{1}
\mathcal{B}_{A}\varphi_{\lambda}\left(x,\omega\right) & =\int_{\mathbb{R}^{d}}e^{-2\pi i\omega y}\exp\left[-\frac{\pi}{2\lambda}\left(Sy\cdot y+2t\cdot y+r\right)\right]dy\\
& =\int_{\mathbb{R}^{d}}e^{-2\pi i\omega y}\exp\left[-\frac{\pi}{2\lambda}\left(\left(y-h\right)\cdot S\left(y-h\right)+k\right)\right]dy,
\end{alignat*}
where $h=-S^{-1}t$ and $k=r-t\cdot S^{-1}t$. To conclude,
\begin{alignat*}{1}
\mathcal{B}_{A}\varphi_{\lambda}\left(x,\omega\right) & =\int_{\mathbb{R}^{d}}e^{-2\pi i\omega y}\exp\left[-\frac{\pi}{2\lambda}\left(\left(y-h\right)\cdot S\left(y-h\right)+k\right)\right]dy\\
& =e^{-\frac{\pi}{2\lambda}k}e^{-2\pi i\omega h}\int_{\mathbb{R}^{d}}e^{-2\pi i\omega y}e^{-\frac{\pi}{2\lambda}ySy}dy\\
& =\left(2\lambda\right)^{d/2}e^{-\frac{\pi}{2\lambda}k}e^{-2\pi i\omega h}\int_{\mathbb{R}^{d}}e^{-2\pi i\sqrt{2\lambda}\omega y}e^{-\pi ySy}dy\\
& =\left(2\lambda\right)^{d/2}\det\left(S\right)^{-1/2}e^{-\frac{\pi}{2\lambda}k}e^{-2\pi i\omega h}e^{-2\pi\lambda\omega S^{-1}\omega}, 
\end{alignat*}
where in the last step we used \cite[App. A - Theorem 1]{folland}.

Therefore, $\mathcal{B}_{A}\varphi_{\lambda}$ turns out to
be a generalized Gaussian function in \eqref{BAgauss}.
\end{proof}
\begin{remark} (i). Notice that Woodbury matrix identity (cf. for instance \cite[Eq. (2.1.4)]{golub}) gives 
	\[
	R^{-1}=\left(I+4MM^{\top}\right)^{-1}=I-M\left(\frac{1}{4}I+M^{\top}M\right)^{-1}M^{\top}=I-4MS^{-1}M^{\top},
	\]
	hence
	\[
	\frac{k}{2}=2x^{2}-8x\cdot MS^{-1}M^{\top}x=2x\cdot\left(I-4MS^{-1}M^{\top}\right)x=2x\cdot R^{-1}x.
	\]
	Furthermore, we see that after setting $z=\left(x,\omega\right)\in\mathbb{R}^{2d}$
	we can write 
	\[
	\mathcal{B}_{A}\varphi_{\lambda}\left(x,\omega\right)=\left(2\lambda\right)^{d/2}\det\left(S\right)^{-1/2}e^{-\pi z\cdot \Sigma z},
	\]
	where 
	\[
	\Sigma=\left(\begin{array}{cc}
	\frac{2}{\lambda}R^{-1} & -4iMS^{-1}\\
	-4iS^{-1}M^{\top} & 2\lambda S^{-1}
	\end{array}\right) \in \mathrm{GL}(2d,\bC).
	\]
	
	The cumbersome way $M$ comes across in $\cB_A \varphi_{\lambda}$ is in fact widely simplified in the case of $\tau$-Wigner distribution, namely for $M=(\tau-1/2)I$, $\tau \in [0,1]$, see \cite[Lemma 2.8]{cdet18}.
	
	\noindent
	
	(ii)  The expression of the Cohen's kernel \eqref{thetaM} can be
	rephrased in more general terms. In fact, it is easy to see that
	\[
	e^{2\pi i\eta\cdot M\xi}=e^{2\pi i\zeta \cdot  Q_{M}\zeta},\qquad\zeta=\left(\xi,\eta\right)\in\mathbb{R}^{2d},
	\]
	where 
	\[
	Q_{M}=\left(\begin{array}{cc}
	0 & M_{sym}\\
	M_{sks} & 0
	\end{array}\right)\in\mathbb{R}^{2d\times2d},
	\]
	and $M_{sym}$ and $M_{sks}$ are the symmetric and skew-symmetric
	parts of $M$ respectively. 
	On the other hand, any block matrix with non-null off-diagonal blocks
	such as
	\[
	Q=\left(\begin{array}{cc}
	0 & V\\
	U & 0
	\end{array}\right)\in\mathbb{R}^{2d\times2d},\qquad U,V\in\mathbb{R}^{d\times d},
	\]
	can be associated with a Cohen-type matrix $A_{M_Q}$, with $M_{Q}=U+V^{\top}$. 
	\end{remark}

\subsection{Time-frequency properties of perturbed representations}

The explicit determination of the Cohen's kernel for a distribution
of Wigner type allows to derive at once a number of important properties
by simply inspecting its analytic expression. To this aim, notice that the Fourier transform of $\theta_{M}$ is
\begin{equation}\label{ThM}
\Theta_{M}\left(\xi,\eta\right)\coloneqq\mathcal{F}\theta_{M}\left(\xi,\eta\right)=\mathcal{F}_{\sigma}\theta_{M}\left(-J(\xi,\eta)\right)=\chi_{M}\left(-\eta,\xi\right)=e^{-2\pi i\xi\cdot M\eta}.
\end{equation}

 It is then clear that the relation between two distributions of the type \eqref{BAM} can be expressed by a Fourier multiplier as follows.

\begin{lemma}
	Let $A_{1}=A_{M_{1}}$ and $A_{2}=A_{M_{2}}$ be two Cohen-type matrices as in \eqref{AM}. Then, 
	\[
	\mathcal{F}\mathcal{B}_{A_{2}}\left(f,g\right)\left(\xi,\eta\right)=e^{-2\pi i\xi\cdot\left(M_{2}-M_{1}\right)\eta}\mathcal{F}\mathcal{B}_{A_{1}}\left(f,g\right)\left(\xi,\eta\right).
	\]
	
	Furthermore, if $M_{2}-M_{1}\in\mathrm{GL}\left(d,\mathbb{R}\right)$,
	\[
	\mathcal{B}_{A_{2}}\left(f,g\right)\left(x,\omega\right)=\frac{1}{\left|\det\left(M_{2}-M_{1}\right)\right|}e^{2\pi ix\cdot\left(M_{2}-M_{1}\right)^{-1}\omega}*\mathcal{B}_{A_{1}}\left(f,g\right)\left(x,\omega\right).
	\]
\end{lemma}
\begin{proof}
	It is a straightforward computation. We leave the details to the interested reader.
\end{proof}

\begin{proposition}\label{bam prop}
Assume that $\mathcal{B}_{A}$ belongs to the Cohen's class. For any $f,g \in \cS(\rd)$, the following
properties are satisfied:
\begin{enumerate}[label=(\roman*)]
\item \textbf{Correct marginal densities}: 
\[
\int_{\mathbb{R}^{d}}\mathcal{B}_{A}f\left(x,\omega\right)d\omega=\left|f\left(x\right)\right|^{2},\qquad\int_{\mathbb{R}^{d}}\mathcal{B}_{A}f\left(x,\omega\right)dx=\left|\hat{f}\left(\omega\right)\right|^{2},\qquad \forall x,\omega \in \rd.
\]
In particular, the energy is preserved:
\[
\iint_{\rdd}\mathcal{B}_{A}f\left(x,\omega\right)dxd\omega=\left\Vert f\right\Vert _{L^{2}}^{2}.
\]
\item \textbf{Moyal's identity}: 
\[
\left\langle \mathcal{B}_{A}f,\mathcal{B}_{A}g\right\rangle _{L^{2}\left(\rdd\right)}=\left|\left\langle f,g\right\rangle \right|^{2}.
\]
\item \textbf{Symmetry}: for all $x,\omega\in\mathbb{R}^{d}$, 
\[
\mathcal{B}_{A}\left(\mathcal{I}f\right)\left(x,\omega\right)=\mathcal{I}\mathcal{B}_{A}f\left(x,\omega\right)=\mathcal{B}_{A}f\left(-x,-\omega\right),
\]
\[
\mathcal{B}_{A}\left(\overline{f}\right)\left(x,\omega\right)=\overline{\mathcal{I}_{2}\mathcal{B}_{A}f\left(x,\omega\right)}=\overline{\mathcal{B}_{A}\left(x,-\omega\right)}.
\]
\item \textbf{Convolution properties}: for all $x,\omega\in\rd$,
\[
\mathcal{B}_{A}\left(f*g\right)\left(x,\omega\right)=\left(\mathcal{B}_{A}f\left(\cdot,\omega\right)*\mathcal{B}_{A}g\left(\cdot,\omega\right)\right)\left(x\right),
\]
\[
\mathcal{B}_{A}\left(f\cdot g\right)(x,\omega)=\left(\mathcal{B}_{A}f\left(x,\cdot\right)*\mathcal{B}_{A}g\left(x,\cdot\right)\right)\left(\omega\right).
\]
\item \textbf{Scaling invariance}: setting $U_{\lambda}f\left(t\right)\coloneqq\left|\lambda\right|^{d/2}f\left(\lambda t\right)$, $\lambda\in\bR\setminus\{0\}$, $t\in\rd$, 
\[
\mathcal{B}_{A}\left(U_{\lambda}f\right)\left(x,\omega\right)=\mathcal{B}_{A}f\left(\lambda x,\lambda^{-1}\omega\right), \qquad \forall x,\omega \in \rd.
\]
\end{enumerate}
\end{proposition}
\begin{proof}
The previous properties can be characterized by requirements on the Fourier transform $\Theta_{M}$ of the corresponding Cohen's kernel, cf.
for instance \cite{cohen tfa 95} and \cite{janssen posspread 97} (for dimension $d=1$
- the stated characterization easily extends to dimension $d>1$):
\begin{enumerate}[label=(\roman*)]
\item $\Theta_{M}\left(0,\omega\right)=\Theta_{M}\left(x,0\right)=1$ for
any $x,\omega\in\mathbb{R}^{d}$ (in particular $\Theta_{M}\left(0,0\right)=1$);
\item $\left|\Theta_{M}\left(x,\omega\right)\right|=1$ for any $x,\omega\in\mathbb{R}^{d}$;
\item $\Theta_{M}\left(-x,-\omega\right)=\Theta_{M}\left(x,\omega\right)$
and $\overline{\Theta_{M}\left(x,\omega\right)}=\Theta_{M}\left(-x,\omega\right)$
respectively, for any $x,\omega\in\mathbb{R}^{d}$;
\item $\Theta_{M}\left(\cdot,\omega_{1}+\omega_{2}\right)=\Theta_{M}\left(\cdot,\omega_{1}\right)\Theta_{M}\left(\cdot,\omega_{2}\right)$
and $\Theta_{M}\left(x_{1}+x_{2},\cdot\right)=\Theta_{M}\left(x_{1},\cdot\right)\Theta_{M}\left(x_{2},\cdot\right)$
respectively, for any $x_{i},\omega_{i}\in\mathbb{R}^{d}$, $i=1,2$.
\item $\Theta_{M}\left(\lambda x,\lambda^{-1}\omega\right)=\Theta_{M}\left(x,\omega\right)$. 
\end{enumerate}
The kernel $\Theta_{M}(x,\omega)=e^{-2\pi i x\cdot M\omega}$ trivially satisfies conditions $(i)$-$(v)$ above.
\end{proof}
\begin{remark} \textbf{Real-valuedness.} Because of Proposition \ref{interchanging f g},
the only real-valued distribution of Wigner type in the Cohen's class
is exactly the Wigner distribution ($M=0$). The condition on $\Theta_{M}$
in order to have this property is indeed $\Theta_{M}\left(x,\omega\right)=\overline{\Theta_{M}\left(-x,-\omega\right)}$. \\
\textbf{Marginal densities.} With little effort, it can be shown
that the marginal densities for a general distribution $\mathcal{B}_{A}$
are given by 
\begin{align*}
\int_{\mathbb{R}^{d}}\mathcal{B}_{A}f\left(x,\omega\right)d\omega&=f\left(A_{11}x\right)\overline{f\left(A_{21}x\right)},\\
\int_{\mathbb{R}^{d}}\mathcal{B}_{A}f\left(x,\omega\right)dx&=\left|\det A\right|^{-1}\hat{f}\left((A^{\#})_{12}\omega\right)\overline{\hat{f}\left(-(A^{\#})_{22}\omega\right)}.
\end{align*}
The correct marginal densities are thus recovered if and only if $A_{11}=A_{21}=I$
and $(A^{\#})_{12}=-(A^{\#})_{22}=I$, and this forces both $\left|\det A\right|=1$
and the block structure of $A$ as that of Cohen's type. This shows
that among the bilinear distributions of Wigner type associated with
invertible matrices, the Cohen-type subclass is made by all and only
those satisfying the correct marginal densities. \\
\textbf{Short-time product formula.} Let us rewrite the STP formula \eqref{STP} for representations
in the Cohen's class: for any $\phi,\psi,f,g\in\mathcal{S}\left(\mathbb{R}^{d}\right)$,
we have
\begin{equation}
V_{\mathcal{B}_{A}\left(\phi,\psi\right)}\mathcal{B}_{A}\left(f,g\right)\left(z,\zeta\right)=e^{-2\pi iz_{2}\zeta_{2}}V_{\phi}f\left(z+P_{M}J\zeta\right)\overline{V_{\psi}g\left(z+\left(I+P_{M}\right)J\zeta\right)},\label{eq:magic cohen}
\end{equation}
where 
\begin{equation}\label{PM} 
  \begin{medsize} P_{M}=\left(\begin{array}{cc}
-\left(M+(1/2)I\right) & 0\\
0 & M-(1/2)I
\end{array}\right),\quad I+P_{M}=\left(\begin{array}{cc}
-\left(M-(1/2)I\right) & 0\\
0 & M+(1/2)I
\end{array}\right).
\end{medsize}
\end{equation} \\ 
\textbf{Covariance formula.} For any $z=\left(z_{1},z_{2}\right),\,w=\left(w_{1},w_{2}\right)\in\rdd$, the covariance formula \eqref{covform} reads

\begin{equation}
\mathcal{B}_{A}\left(\pi\left(z\right)f,\pi\left(w\right)g\right)\left(x,\omega\right)=e^{2\pi i\left[\frac{1}{2}\left(z_{2}+w_{2}\right)+M\left(z_{2}-w_{2}\right)\right]\left(z_{1}-w_{1}\right)}M_{J\left(z-w\right)}T_{\mathcal{T}_{M}\left(z,w\right)}\mathcal{B}_{A}\left(f,g\right)\left(x,\omega\right),\label{eq:covar cohen}
\end{equation}
where 
\begin{align*}
\mathcal{T}_{M}\left(z,w\right) & =\left(\begin{array}{c}
(1/2)\left(z_{1}+w_{1}\right)+M\left(w_{1}-z_{1}\right)\\
(1/2)\left(z_{2}+w_{2}\right)+M\left(z_{2}-w_{2}\right)
\end{array}\right) \\ & =\frac{1}{2}\left(z+w\right)+\left(\begin{array}{cc}
-M & 0\\
0 & M 
\end{array}\right) \left(z-w\right).
\end{align*}

Alternatively, using \eqref{PM},
\[
\mathcal{T}_{M}\left(z,w\right)=\left(\begin{array}{c}
-\left(M-(1/2)I\right)z_{1}+\left(M+(1/2)I\right)w_{1}\\
\left(M+(1/2)I\right)z_{2}-\left(M-(1/2)I\right)w_{2}
\end{array}\right)=\left(I+P_{M}\right)z-P_{M}w.
\]
\end{remark}

\subsubsection{\bf Support conservation}

A desirable property for a time-frequency distribution is the preservation
of the support of the original signal. A scale of precise mathematical
conditions can be introduced in order to capture this heuristic feature.
Following Folland's classic approach (see \cite[p. 59]{folland}),
in this section we define the support of a signal $f:\mathbb{R}^{d}\rightarrow\mathbb{C}$
as the smallest closed set outside of which $f=0$ a.e., hence we
may assume $f\equiv0$ everywhere outside $\mathrm{supp}f$.
\begin{definition}
Let $Qf:\mathbb{R}_{\left(x,\omega\right)}^{2d}\rightarrow\mathbb{C}$
be the time-frequency distribution associated to the signal $f:\mathbb{R}_{t}^{d}\rightarrow\mathbb{C}$
in a suitable function space. Let $\pi_{x}:\mathbb{R}_{\left(x,\omega\right)}^{2d}\rightarrow\mathbb{R}_{x}^{d}$
and $\pi_{\omega}:\mathbb{R}_{\left(x,\omega\right)}^{2d}\rightarrow\mathbb{R}_{\omega}^{d}$
be the projections onto the first and second factors ($\mathbb{R}_{\left(x,\omega\right)}^{2d}\simeq\mathbb{R}_{x}^{d}\times\mathbb{R}_{\omega}^{d}$)
and, for any $E\subset\mathbb{R}^{d}$, let $\mathcal{C}\left(E\right)$
denote the closed convex hull of $E$. 
\begin{itemize}
\item $Q$ satisfies the time strong support property if 
\[
f\left(x\right)=0\Leftrightarrow Qf\left(x,\omega\right)=0\qquad\forall\omega\in\mathbb{R}^{d}.
\]
$Q$ satisfies the frequency strong support property if 
\[
\hat{f}\left(\omega\right)=0\Leftrightarrow Qf\left(x,\omega\right)=0\qquad\forall x\in\mathbb{R}^{d}.
\]
\item $Q$ satisfies the time weak support property if 
\[
\pi_{x}\left(\mathrm{supp}Qf\right)\subset\mathcal{C}\left(\mathrm{supp}f\right).
\]
$Q$ satisfies the frequency weak support property if 
\[
\pi_{\omega}\left(\mathrm{supp}Qf\right)\subset\mathcal{C}\left(\mathrm{supp}\hat{f}\right).
\]
\end{itemize}
We say that $Q$ satisfies the strong (resp. weak) support property
if both time and frequency strong (resp. weak) support properties
hold. 
\end{definition}
We restrict our attention to MWDs in the Cohen's class (i.e., $\cB A$ with $A=A_M$ as in \eqref{AM}) and completely
characterize those satisfying the aforementioned properties, showing
the optimality in this sense of $\tau$-Wigner distributions. 
\begin{theorem}
The only MWDs in Cohen's class satisfying the strong correct support
properties are Rihaczek and conjugate-Rihaczek distributions. 
\end{theorem}
\begin{proof}
This result can be inferred by directly inspecting the Fourier transform of Cohen's kernel. Indeed, by adapting the proof of Janssen (see \cite[Sec. 2.6.2]{janssen posspread 97}) to dimension $d>1$ one can show that the
only members of the Cohen's class satisfying both time and frequency
strong support property are linear combinations of Rihackez and conjugate-Rihaczek distributions. This is equivalent to the following condition on the Fourier transform of the kernel $\Theta_{M}$: for any $x,\omega \in \rd$,

\[
\Theta_{M}(x,\omega)=C_{+}e^{\pi i x\omega} + C_{-}e^{-\pi i x\omega},
\]
for some $C_{+},C_{-}\in\bC$. Since $\Theta_{M}$ has the form \eqref{ThM}, this can happen if and only if $M=\pm(1/2)I$ with $C_{+}=1,0$ and $C_{-}=0,1$ respectively.
\end{proof}
\begin{theorem}
Let $A=A_{M}\in\GLL$ be a Cohen-type
matrix. The only associated distributions satisfying the weak support
property are the $\tau$-Wigner distributions, namely 
\[
M=\left(\tau-\frac{1}{2}\right)I,\qquad\tau\in\left[0,1\right].
\]
\end{theorem}
\begin{proof}
Assume $x\in\mathrm{supp}\mathcal{B}_{A}f\left(\cdot,\omega\right)$
for a fixed $\omega\in\mathbb{R}^{d}$. The only way for this to happen
is to have
\[
f\left(x+\left(M+\frac{1}{2}I\right)y\right)\ne0,\qquad f\left(x+\left(M-\frac{1}{2}\right)y\right)\ne0,
\]
hence $x+\left(M+\frac{1}{2}I\right)y,\,x+\left(M-\frac{1}{2}I\right)y\in\mathrm{supp}f$.
In order to have $x\in\mathcal{C}\left(\mathrm{supp}f\right)$, we
require that 
\[
x=\lambda\left(x+\left(M+\frac{1}{2}I\right)y\right)+\mu\left(x+\left(M-\frac{1}{2}I\right)y\right),
\]
for some $\lambda,\mu\ge0$ such that $\lambda+\mu=1$. Rewriting
this condition as 
\[
x=\left(\lambda+\mu\right)x+\left(\frac{1}{2}\left(\lambda-\mu\right)I+\left(\lambda+\mu\right)M\right)y,
\]
gives the constraints
\[
\lambda+\mu=1,\qquad\frac{1}{2}\left(\lambda-\mu\right)I+M=0.
\]
Therefore, suitable solutions exist if and only if 
\[
M=\left(\tau-\frac{1}{2}\right)I,\qquad0\le\tau\le1.
\]

Similar arguments for the frequency weak support property shall be
applied to $\mathcal{B}_{A}\hat{f}\left(x,\omega\right)=\mathcal{B}_{A'}f\left(-\omega,x\right),$where
\[
A'=\mathcal{I}_{2}A^{\#}\tilde{I}=\left(\begin{array}{cc}
I & M-(1/2)I\\
I & -\left(M+(1/2)I\right)
\end{array}\right).
\]
\end{proof}

\subsection{Time-frequency analysis of the kernel}

In this section we deepen the study of the Cohen's kernel $\theta_{M}$
by introducing a fine scale of functional spaces with specific resolution
of the time-frequency content of $\theta_{M}$, following the approach
of \cite[Proposition. 4.1]{cdgn tfa bj} for the Cohen's kernels for $\tau$-Wigner
distributions, which will be in fact recovered below. Hereinafter
we assume $M\in\GLL$ if not specified
otherwise. 

Recall that $\Theta_{M}=\cF \theta_{M}$, where (cf. \eqref{ThM})
\[
\Theta_{M}\left(\xi,\eta\right)=\mathcal{F}\theta_{M}\left(\xi,\eta \right)=e^{-2\pi i\xi\cdot M\eta}\left(\in\mathcal{S}'\left(\rdd\right)\right).
\]
At a first glance we notice that $\Theta_{M}\in C^{\infty}\left(\rdd\right)\cap L^{\infty}\left(\rdd\right)$
and $\Theta_{M}\in L_{\text{loc}}^{p}\left(\rdd\right)$
for any $1\leq p\leq\infty$. Hence, we are dealing with distributions
whose Fourier transforms are well-behaved dilated chirps, and intuition
suggests that the kernels themselves should belong to the same family.
This heuristic statement is enforced by the following result, already
proved in \cite[Proposition 3.2 and Corollary 3.4]{cdgn red int}. 
\begin{lemma}
\label{chirp spaces}The function $\Theta\left(x,\omega\right)=e^{2\pi ix\omega}$ belongs to $M^{1,\infty}\left(\rdd\right)\cap W\left(\mathcal{F}L^{1},L^{\infty}\right)\left(\rdd\right)$. 
\end{lemma}
Using this issue and dilation properties for Wiener amalgam spaces, we infer 

\begin{proposition}
\label{kernel spaces}Let $A=A_{M}\in\GLL$
be a Cohen-type matrix with $M\in\mathrm{GL}\left(d,\mathbb{R}\right)$.
We have 
\[
\theta_{M}\in M^{1,\infty}\left(\rdd\right)\cap W\left(\mathcal{F}L^{1},L^{\infty}\right)\left(\rdd\right).
\]
\end{proposition}
\begin{proof}
Notice that
\[
\Theta_{M}\left(\xi,\eta\right)=e^{-2\pi i\xi\cdot M\eta}=D_{\tilde{M}}\Theta\left(\xi,\eta\right),
\]
where $\Theta\left(\xi,\eta\right)=e^{2\pi i\xi\eta}$ and $D_{Q}$ is the dilation operator $D_{Q}f\left(t\right)\coloneqq f\left(Qt\right)$
associated with an invertible matrix $Q\in\GLL$,
in particular 
\[
\qquad\tilde{M}=\left(\begin{array}{cc}
-I & 0\\
0 & M
\end{array}\right).
\]
It is clear that $\tilde{M}$ is invertible if and only if $M$ is
invertible. 

Therefore, according to the dilation properties in \cite[Proposition 3.1 and Corollary 3.2]{cn met rep},
the results in \cite[Proposition 3.2]{cdgn red int} and Lemma \ref{chirp spaces},
we have $\Theta_{M}\in M^{1,\infty}\left(\rdd\right)\cap W\left(\mathcal{F}L^{1},L^{\infty}\right)$.
Since $\Theta_{M}=\mathcal{F}\theta_{M}$ and $W\left(\mathcal{F}L^{1},L^{\infty}\right)=\mathcal{F}\left(M^{1,\infty}\right)$,
we conclude that $\theta_{M}\in M^{1,\infty}\left(\rdd\right)\cap W\left(\mathcal{F}L^{1},L^{\infty}\right)\left(\rdd\right)$. 
\end{proof}
In order to compute the expression of $\theta_{M}$, it seems useful
to recall some facts concerning dilations, tempered distributions
and Fourier transform. Given an invertible matrix $A\in\mathrm{GL}\left(d,\mathbb{R}\right)$
and a tempered distribution $u\in\mathcal{S}'\left(\mathbb{R}^{d}\right)$,
the dilated distribution $D_{A}u\in\mathcal{S}'\left(\mathbb{R}^{d}\right)$
is defined as follows:
\[
\left\langle D_{A}u,\phi\right\rangle \coloneqq\left\langle u,\left|\det A\right|^{-1}D_{A^{-1}}\phi\right\rangle ,\qquad\forall\phi\in\mathcal{S}\left(\mathbb{R}^{d}\right).
\]
The behaviour of the Fourier transform under dilations is given by
the following formula:
\[
\mathcal{F}D_{A}\phi=\left|\det A\right|^{-1}D_{A^{\#}}\mathcal{F}\phi,\qquad\phi\in\mathcal{S}\left(\mathbb{R}^{d}\right),
\]
where $A^{\#}=\left(A^{-1}\right)^{\top}$ as usual. 

The validity of the following relation can be verified by a direct
computation: for any $u\in\mathcal{S}'\left(\mathbb{R}^{d}\right)$
and $A\in\mathrm{GL}\left(d,\mathbb{R}\right)$,
\[
D_{A}\mathcal{F}^{-1}u=\left|\det A\right|^{-1}\mathcal{F}^{-1}D_{A^{\#}}u.
\]

It is now enough to notice that, according to the notation employed in the proof
of Proposition \ref{kernel spaces}, 
\[
\theta_{M}=\mathcal{F}^{-1}\Theta_{M}=\mathcal{F}^{-1}D_{\tilde{M}}\Theta=\left|\det M\right|^{-1}D_{\tilde{M}^{\#}}\mathcal{F}^{-1}\Theta.
\]
A short computation concludes the proof of the following result, which
confirms the initial intuition. 
\begin{theorem}\label{intro2}
Let $A=A_{M}\in\GLL$ be a Cohen-type
matrix with $M\in\mathrm{GL}\left(d,\mathbb{R}\right)$. Then, the kernel $\theta_M$ is given by \eqref{thM}.
\end{theorem}

\begin{proof}
We have
\[
\theta_{M}=\mathcal{F}^{-1}D_{\tilde{M}}\Theta=\left|\det M\right|^{-1}D_{\left(\tilde{M}^{\top}\right)^{-1}}\mathcal{F}^{-1}\Theta.
\]
Notice that
\[
\mathcal{F}^{-1}\Theta=\mathcal{I}\mathcal{F}\Theta=\mathcal{F}\Theta=D_{J}\Theta,
\]
so that
\[
\theta_{M}=\left|\det M\right|^{-1}D_{\left(\tilde{M}^{\top}\right)^{-1}J}\Theta,
\] with
\[
\left(\tilde{M}^{\top}\right)^{-1}J=\left(\begin{array}{cc}
0 & -I\\
-\left(M^{\top}\right)^{-1} & 0
\end{array}\right).
\]

\end{proof}
In particular, in the case of $\tau$-Wigner distributions, namely
$M=M_{\tau}=\left(\tau-\frac{1}{2}\right)I$, $\tau\in\left[0,1\right]\backslash\left\{ \tfrac{1}{2}\right\} $,
we recover a known result (see for instance \cite[Proposition 5.6]{bogetal}):
\[
\theta_{M_{\tau}}\left(x,\omega\right)=\frac{2^{d}}{\left|2\tau-1\right|^{d}}e^{2\pi i\frac{2}{2\tau-1}x\omega}.
\]

Notice that one cannot say much without assuming the invertibility
of $M$. We do not explore this situation, apart from mentioning that
for $M=0$ most of these results do not hold: for instance, since
$\theta_{0}=\delta$, it is easy to verify that $\theta_{0}\in M^{1,\infty}\left(\rdd\right)\backslash W\left(\mathcal{F}L^{1},L^{\infty}\right)\left(\rdd\right)$, cf. \cite{cdgn tfa bj}.

To conclude this section we prove that, in according with heuristic
expectations, linear perturbations of the Wigner distribution yield representations
which share the same smoothness and decay as the Wigner one.
\begin{theorem}
	Let $A=A_{M}\in\GLL$ be a matrix of
	Cohen's type with $M\in\mathrm{GL}\left(d,\mathbb{R}\right)$ and $f\in\mathcal{S}'\left(\mathbb{R}^{d}\right)$
	be a signal. Then, for $1\leq p,q\leq\infty$, we have 
	\[
	Wf\in M^{p,q}\left(\rdd\right) \Longleftrightarrow\mathcal{B}_{A_{M}}f\in M^{p,q}\left(\rdd\right).
	\]	
\end{theorem}
\begin{proof}
	Assume first $Wf\in M^{p,q}\left(\rdd\right)$, for some $1\leq p,q\leq \infty$. Taking the symplectic Fourier transform, this is equivalent to showing
	that 
	\[
	\theta_{M}\cdot Amb\left(f\right)\in W\left(\mathcal{F}L^{p},L^{q}\right)\left(\rdd\right).
	\]
	Notice that $Amb\left(f\right)\in W\left(\mathcal{F}L^{p},L^{q}\right)\left(\rdd\right)$
	because of the assumption on $Wf$. The claim thus follows from the
	known fact that $W\left(\mathcal{F}L^{p},L^{q}\right)\left(\rdd\right)$
	is a Banach module over $W\left(\mathcal{F}L^{1},L^{\infty}\right)\left(\rdd\right)$
	(cf. \cite[Theorem 1]{fei was prod}), namely 
	\[
	W\left(\mathcal{F}L^{1},L^{\infty}\right)\left(\rdd\right)\cdot W\left(\mathcal{F}L^{p},L^{q}\right)\left(\rdd\right)\subset W\left(\mathcal{F}L^{p},L^{q}\right)\left(\rdd\right),
	\]
	and from Proposition \ref{kernel spaces}, yielding $\theta_{M}\in W\left(\mathcal{F}L^{1},L^{\infty}\right)\left(\rdd\right)$. \par
	Vice versa, assume $\mathcal{B}_{A_{M}}f\in M^{p,q}\left(\rdd\right)$, for some $1\leq p,q\leq \infty$.  Taking the symplectic Fourier transform,
	$$
	\theta_{M}\cdot Amb\left(f\right)=\cF_\sigma  \mathcal{B}_{A_{M}}f\in  W\left(\mathcal{F}L^{p},L^{q}\right)\left(\rdd\right),
	$$
	that is, 
	$$
	Amb\left(f\right)=\theta_{M}^{-1}\cF_\sigma  \mathcal{B}_{A_{M}}f\in  W\left(\mathcal{F}L^{1},L^{\infty}\right)\cdot W\left(\mathcal{F}L^{p},L^{q}\right)\subset W\left(\mathcal{F}L^{p},L^{q}\right).
	$$
	Indeed, by \eqref{thM},  the function $\theta_{M}^{-1}$ is given by
	\begin{equation*}
	\theta_{M}^{-1}\left(x,\omega\right)=\left|\det M\right| e^{-2\pi ix\cdot M^{-1}\omega}
	\end{equation*}
	and satisfies $\theta_{M}^{-1}\in W\left(\mathcal{F}L^{1},L^{\infty}\right)\left(\rdd\right),$ since 
	$e^{-2\pi ix\cdot M^{-1}\omega}=\overline{e^{2\pi ix\cdot M^{-1}\omega}}$ and the chirp
	$e^{2\pi ix\cdot M^{-1}\omega}\in W\left(\mathcal{F}L^{1},L^{\infty}\right)\left(\rdd\right)$, by Proposition \ref{kernel spaces}.
	This concludes the proof.
\end{proof}
\begin{remark}
Similar results have been proved for the Born-Jordan distribution
in \cite[Theorem 4.1]{cdgn red int} and its $n$-th order generalization in \cite{cdgdn gen bj}, although with a significant difference:
no directional smoothing effects occur in our scenario. The subsequent
section, devoted to the study of interferences, will present an improvement in this direction.
\end{remark}

\subsection{A study of interferences}

In a broad sense, interferences are artefacts occurring when non-zero values of the representation appears
in regions of the phase space where the signal contains no energy.
In view of the applications, it is obviously desirable to reduce the
occurrence of these phenomena but the literature shows that this aim
can be accomplished only at the expenses of other possibly relevant
properties. In this spirit, it has been recently proved that the effectiveness
of interference damping is subtly related to covariance of the representation
with respect to symplectic transformations of the phase space, see
\cite{cdgdn sympcov interf} for details. In particular,
in view of \cite[Proposition 4.4 and Theorem 4.6]{cdgdn sympcov interf}, we
remark that full symplectic covariance is one of the properties of
the Wigner distribution which are lost under the effect of linear
perturbations. It can be still interesting to determine partial symmetries,
i.e., covariance with respect to certain subgroups of $\mathrm{Sp}\left(2d,\mathbb{R}\right)$.
In our case, however, this will also depend on $M$: for instance,
given 
\[
V_{P}=\left(\begin{array}{cc}
I & 0\\
-P & I
\end{array}\right),\qquad P\in\mathbb{R}^{d\times d}\text{ symmetric},
\]
(see \cite[Proposition 62 and Corollary 63]{dg sympmeth} for its role in symplectic
algebra) we see that $\theta_{M}\circ V_{P}=\theta_{M}$ for any $P$
if and only if $M^{-1}Px\cdot x=0$ for any $x\in\mathbb{R}^{d}$,
that is if $M$ is skew-symmetric.

As this argument suggests, the suppression of interferences cannot
be effectively performed by means of linear perturbations. In order
to experience this, we limit ourselves to dimension $d=1$ and follow
the geometrical approach employed in \cite{bco quadratic}. As a toy
model we consider signals consisting of pure frequencies confined
in disjoint time intervals. It is well known that the Wigner transform
displays ``ghost frequencies'' in between any couple of true frequencies
of the signal. A similar phenomenon can be studied also in higher
dimension considering Gaussian signals in the so-called ``diamond
configuration'', see again \cite{cdgdn sympcov interf}. 

We remark that for $d=1$ the perturbation matrix $M$ boils down
to a scalar $m\in\mathbb{R}$. Let $f$ be a signal with a frequency
$\omega_{1}$ appearing in the interval $I_{1}=\left[x_{1},x_{1}+h_{1}\right]$
and $\omega_{2}$ in $I_{2}=\left[x_{2},x_{2}+h_{2}\right]$, with
$h_{2}\ge h_{1}>0$ such that $x_{1}+h_{1}<x_{2}$. The distribution in \eqref{BAM} becomes 
\[
\mathcal{B}_{m}f\left(x,\omega\right)=\int_{\mathbb{R}}e^{-2\pi i\omega y}f\left(x+\left(m+\frac{1}{2}\right)y\right)\overline{f\left(x+\left(m-\frac{1}{2}\right)y\right)}dy.
\]
We see that $\mathcal{B}_{m}f$ is supported in the diamond-shaped
regions $D_{i}$, $i=1,...,4$, (see Figure \ref{fig1}) obtained by
intersecting the following straight lines passing through the endpoints of the time intervals:
\[
\begin{cases}
x+\left(m\pm\frac{1}{2}\right)y=x_{1}\\
x+\left(m\pm\frac{1}{2}\right)y=x_{1}+h_{1}\\
x+\left(m\pm\frac{1}{2}\right)y=x_{2}\\
x+\left(m\pm\frac{1}{2}\right)y=x_{2}+h_{2}.
\end{cases}
\]

\begin{figure}
\begin{centering}
\includegraphics[scale=0.6]{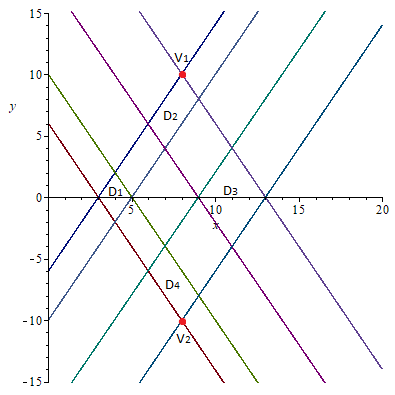}
\par\end{centering}
\caption{Support of $\mathcal{B}_{m}f$ with $m=0$, $I_{1}=\left[3,5\right]$,
$I_{2}=\left[9,13\right]$.} \label{fig1}
\end{figure}
With the notation of the figure, we see that $D_{1}$ and $D_{3}$
give account for the true frequencies of the signal, while $D_{2}$
and $D_{3}$ are non-zero interferences. A short computation shows
that the coordinates of the two points $V_{1}$ and $V_{2}$ are 
\[
V_{1}=\left(\frac{2m+1}{2}\left(x_{2}+h_{2}\right)-\frac{2m-1}{2}x_{1},x_{1}-\left(x_{2}+h_{2}\right)\right),
\]
\[
V_{2}=\left(\frac{2m+1}{2}x_{1}-\frac{2m-1}{2}\left(x_{2}+h_{2}\right),\left(x_{2}+h_{2}\right)-x_{1}\right),
\]
hence we see that the only effect of the perturbation parameter $m$
is the horizontal translation of the diamond's corners, giving no
room for damping. The only reduction procedure that can still be performed
is the one proposed by Boggiatto et al. in \cite{bco quadratic},
even if its validity is restricted to the special class of signals
examined insofar. Furthermore, notice that we are in fact studying
re-parametrized $\tau$-Wigner distributions in a broad sense, since
now $m=\tau-(1/2)$ is free to run over $\mathbb{R}$. As already
seen before and also noticeable from the coordinates of $V_{1}$ and
$V_{2}$, when $m\in\mathbb{R}\setminus\left[-\tfrac{1}{2},\tfrac{1}{2}\right]$,
the support of the signal is no longer conserved - neither in weak
sense.

To conclude this section, we notice that an efficient reduction of
interferences requires the Cohen's kernel to show some decay at infinity
- which is not the case of the chirp-like kernel $\theta_{M}$. In
order to enhance this feature and at the same time to not lose other
relevant ones satisfied by MWDs in the Cohen's class, it seems reasonable
to consider the Cohen's distributions with kernels of type $\theta_{M}*\varphi$,
where $\varphi\in\mathcal{S}'\left(\rdd\right)$ is a decaying
distribution satisfying suitable properties - for instance, one may
ask that $\hat{\varphi}\left(\zeta_{1},\zeta_{2}\right)=\Phi\left(\zeta_{1}\cdot\zeta_{2}\right)$
with $\Phi\left(0\right)=1$ in order to keep the correct marginal
densities. This kind of investigation deserves a special treatment
in order to balance the trade-off between theoretically relevant features
and practical purposes, thus it cannot be provided here. We confine
ourselves to mention that choosing as smoothing distribution the one
corresponding to the Born-Jordan kernel, namely
\[
\varphi_{\sigma}\left(\zeta\right)=\hat{\varphi}\left(J\zeta\right)=\frac{\sin\left(\pi\zeta_{1}\zeta_{2}\right)}{\pi\zeta_{1}\zeta_{2}},\qquad\zeta=\left(\zeta_{1},\zeta_{2}\right)\in\rdd,
\]
allows to enjoy some smoothing phenomena recently investigated, cf. for example \cite[Theorem 4.1]{cdgn red int}. 

\subsection{Continuity on functional spaces}

In this section we prove that the continuity of bilinear distributions
associated with matrices of Cohen's type on modulation and Wiener
amalgam spaces is a property stable under linear perturbations. We work with weights of polynomial type as in \eqref{vs}. 
\begin{theorem}\label{sharpbou}
Let $A=A_{M}\in\mathrm{GL}\left(2d,\mathbb{\mathbb{R}}\right)$ be a
Cohen-type matrix. Let $1\le p_{i},q_{i},p,q<\infty$, $i=1,2$, such
that 
\begin{equation}
p_{i},q_{i}\le q,\qquad i=1,2,\label{eq:cond1sharp}
\end{equation}
and 
\begin{equation}
\frac{1}{p_{1}}+\frac{1}{p_{2}}\ge\frac{1}{p}+\frac{1}{q},\qquad\frac{1}{q_{1}}+\frac{1}{q_{2}}\ge\frac{1}{p}+\frac{1}{q}.\label{eq:cond2sharp}
\end{equation}

\begin{enumerate}[label=(\roman*)]
	\item If $f_{1}\in M_{v_{\left|s\right|}}^{p_{1},q_{1}}\left(\mathbb{R}^{d}\right)$
	and $f_{2}\in M_{v_{s}}^{p_{2},q_{2}}\left(\mathbb{R}^{d}\right)$,
	then $\mathcal{B}_{A}\left(f_{1},f_{2}\right)\in M_{1\otimes v_{s}}^{p,q}\left(\mathbb{R}^{2d}\right)$,
	and the following estimate holds:
	\[
	\left\Vert \mathcal{B}_{A}\left(f_{1},f_{2}\right)\right\Vert _{M_{1\otimes v_{s}}^{p,q}}\lesssim_{M}\left\Vert f_{1}\right\Vert _{M_{v_{\left|s\right|}}^{p_{1},q_{1}}}\left\Vert f_{2}\right\Vert _{M_{v_{s}}^{p_{2},q_{2}}}.
	\] {
	\item Assume further that both $M-\left(1/2\right)I$ and $M+\left(1/2\right)I$
	are invertible (equivalently: $A_M$ is right-regular, or $P_{M}$ is invertible, cf. \eqref{PM}). Set $v_{s}^M \equiv v_{s}\circ (I+P_M^{-1})^{-1}$. If $f_{1}\in M_{v_{\left|s\right|}}^{p_{1},q_{1}}\left(\mathbb{R}^{d}\right)$
	and $f_{2}\in M_{v_{s}^M}^{p_{2},q_{2}}\left(\mathbb{R}^{d}\right)$,
	then $\mathcal{B}_{A}\left(f_{1},f_{2}\right)\in W\left(\mathcal{F}L^{p},L_{v_{s}^M}^{q}\right)\left(\mathbb{R}^{2d}\right)$,
	and the following estimate holds:
	\[
	\left\Vert \mathcal{B}_{A}\left(f_{1},f_{2}\right)\right\Vert _{W\left(\mathcal{F}L^{p},L_{v_{s}^M}^{q}\right)}\lesssim_{M}\left(C_{M}\right)^{1/q-1/p}\left\Vert f_{1}\right\Vert _{M_{v_{\left|s\right|}}^{p_{1},q_{1}}}\left\Vert f_{2}\right\Vert _{M_{v_{s}^M}^{p_{2},q_{2}}},
	\]
	where 
	\begin{equation}\label{CM}
	C_{M}=\left|\det\left(M+\frac{1}{2}I\right)\det\left(M-\frac{1}{2}I\right)\right|>0.
\end{equation}}
\end{enumerate}
\end{theorem}
\begin{proof}
Fix $g\in\mathcal{S}\left(\mathbb{R}^{d}\right)$
and set $\Phi=\mathcal{B}_{A}g\in\mathcal{S}\left(\rdd\right)$.
The key insight here is given by the short-time product formula in \eqref{STP}. Precisely, for any $z,\zeta\in\rdd$
we have
\[
\left|V_{\Phi}\mathcal{B}_{A}\left(f_{1},f_{2}\right)\left(z,\zeta\right)\right|=\left|V_{g}f_{1}\left(z+P_{M}J\zeta\right)\overline{V_{g}f_{2}\left(z+\left(I+P_{M}\right)J\zeta\right)}\right|,
\]
where $P_{M}$ is the matrix defined in \eqref{PM}. Consequently, for $p,q<\infty$,
\begin{align*}
& \left\Vert \mathcal{B}_{A}\left(f_{1},f_{2}\right)\right\Vert _{M_{1\otimes v_{s}}^{p,q}}\\
\asymp & \left(\int_{\rdd}\left(\int_{\rdd}\left|V_{g}f_{1}\left(z+P_{M}J\zeta\right)\right|^{p}\left|V_{g}f_{2}\left(z+\left(I+P_{M}\right)J\zeta\right)\right|^{p}dz\right)^{q/p}v_{s}\left(J\zeta\right)^{q}d\zeta\right)^{1/q}.
\end{align*}
The change of variables $z\mapsto z-\left(I+P_{M}\right)J\zeta$ turns
the integral over $z$ into a convolution, then
\begin{align*}
\left\Vert \mathcal{B}_{A}\left(f_{1},f_{2}\right)\right\Vert _{M_{1\otimes v_s}^{p,q}} & \asymp\left(\int_{\rdd}\left(\left|V_{g}f_{2}\right|^{p}*\left|\left(V_{g}f_{1}\right)^{*}\right|^{p}\right)^{q/p}\left(J\zeta\right)v_{s}\left(J\zeta\right)^{q}d\zeta\right)^{1/q}\\
 & =\left\Vert \left|V_{g}f_{2}\right|^{p}*\left|\left(V_{g}f_{1}\right)^{*}\right|^{p}\right\Vert _{L_{v_{ps}}^{q/p}}^{1/p}.
\end{align*}

From now on, the proof proceeds exactly as in \cite[Theorem 3.1]{cn sharp int}.
Similar arguments also hold whenever  $p=\infty$ or $q=\infty$.

{For what concerns boundedness on Wiener amalgam spaces, notice first
that if $A_{M}$ is right regular, $\left(I+P_{M}\right)$ and also
$\left(I+P_{M}^{-1}\right)$ are invertible, with
\begin{equation}\label{pm}
\left(I+P_{M}^{-1}\right)^{-1}=\left(I+P_{M}\right)^{-1}P_{M}.
\end{equation}
Therefore,
\begin{align*}
& \left\Vert \mathcal{B}_{A}\left(f_{1},f_{2}\right)\right\Vert _{W\left(\mathcal{F}L^{p},L_{v_{s}\circ\left(I+P_{M}\right)^{-1}}^{q}\right)}\\
\asymp & \left(\int_{\mathbb{R}^{2d}}\left(\int_{\mathbb{R}^{2d}}\left|V_{g}f_{1}\left(z+P_{M}J\zeta\right)\right|^{p}\left|V_{g}f_{2}\left(z+\left(I+P_{M}\right)J\zeta\right)\right|^{p}d\zeta\right)^{q/p}v_{s}\left(\left(I+P_{M}\right)^{-1}z\right)^{q}dz\right)^{1/q}.
\end{align*}
Using the change of variables $\eta=z+P_{M}J\zeta$ and then the matrix equality \eqref{pm}, we can write
{\small{}
\begin{align*}
& \left\Vert \mathcal{B}_{A}\left(f_{1},f_{2}\right)\right\Vert _{W\left(\mathcal{F}L^{p},L_{v_{s}\circ\left(I+P_{M}\right)^{-1}}^{q}\right)}\\
\asymp & C_{M}^{-1/p}\left(\int_{\mathbb{R}^{2d}}\left(\int_{\mathbb{R}^{2d}}\left|V_{g}f_{1}\left(\eta\right)\right|^{p}\left|V_{g}f_{2}\left(\left(I+P_{M}^{-1}\right)\eta-P_{M}^{-1}z\right)\right|^{p}d\eta\right)^{q/p}\right.\\
&\,\,\qquad\qquad\cdot\,\left. v_{s}\left(\left(I+P_{M}\right)^{-1}z\right)^{q}dz\right)^{1/q}\\
= & C_{M}^{-1/p}\left(\int_{\mathbb{R}^{2d}}\left(\int_{\mathbb{R}^{2d}}\left|V_{g}f_{1}\left(\eta\right)\right|^{p}\left|V_{g}f_{2}\left(\left(I+P_{M}^{-1}\right)\left(\eta-\left(I+P_{M}^{-1}\right)^{-1}P_{M}^{-1}z\right)\right)\right|^{p}d\eta\right)^{q/p}\right.\\
&\,\,\qquad\qquad\cdot\,\left. v_{s}\left(\left(I+P_{M}\right)^{-1}z\right)^{q}dz\right)^{1/q}\\
= & C_{M}^{-1/p}\left(\int_{\mathbb{R}^{2d}}\left(\left|V_{g}f_{1}\right|^{p}*\left|\left((V_{g}f_{2})^*\left(\left(I+P_{M}^{-1}\right)\cdot\right)\right)^{*}\right|^{p}\right)^{q/p}\left(\left(I+P_{M}\right)^{-1}z\right)\right.\\
& \,\,\qquad\qquad\cdot\,\left.  v_{s}\left(\left(I+P_{M}\right)^{-1}z\right)^{q}dz\right)^{1/q}.\\
= & C_{M}^{1/q-1/p}\left\Vert \left|V_{g}f_{1}\right|^{p}*\left|\left((V_{g}f_{2})^*\left(\left(I+P_{M}^{-1}\right)\cdot\right)\right)^{*}\right|^{p}\right\Vert _{L_{v_{ps}}^{q/p}}^{1/p},
\end{align*}}
where  the constant $C_M$ is defined in \eqref{CM} and we write $(V_{g}f_{2})^*(z)=(V_{g}f_{2})(-z)$. Again, the proof proceeds hereinafter as in \cite[Theorem 3.1]{cn sharp int}.
}
\end{proof}
{
\begin{remark}
We remark that the given estimates are not sharp, since we employed window functions depending on $M$ in order to perform the computations and thus the hidden constants in the symbol $\lesssim_{M}$ may depend on $M$. However, the comments of \cite[Remark 3.2]{cn sharp int} are still valid here. In particular, the result holds for more general weight functions:
for instance, sub-exponential weights or polynomial weights satisfying
formula $\left(4.10\right)$ in \cite{toft cont 2} are suitable choices.
Notice that the proof of the theorem in fact reduces to the study of continuity estimates for convolutions on weighted Lebesgue mixed-norm
spaces.
We would also point out that results in the spirit of Theorem \ref{sharpbou}(ii) have been already proved for $\tau$-Wigner distributions in \cite[Lemma 3.1]{cdet18} and \cite{cnt18} and can be easily generalized to MWDs. In particular, we recover \cite[Lemma 4.2]{cnt18} by noticing that $(I+P_M)^{-1}=\cB_{\tau}$ and $(I+P_M^{-1})^{-1}=\cU_{\tau}$ for $M=(\tau-1/2)I$, where the matrices $\mathcal{B}_\tau$ and $\cU_{\tau}$ are defined in \cite[(5) and (26)]{cnt18}.
\end{remark}
}

Under more restrictive conditions on the Cohen-type matrix, namely
assuming right-regularity (hence that both $M-(1/2)I$ and
$M+(1/2)I$ are invertible), we are able to apply Proposition
\ref{right-regular continuity}, obtaining boundedness results on Lebesgue spaces.  
\begin{theorem}
Let $A=A_{M}\in\GLL$ be a right-regular
Cohen-type matrix. For any $1<p<\infty$ and $q\ge2$ such that $q'\le p\le q$, $f\in L^{p}\left(\mathbb{R}^{d}\right)$
and $g\in L^{p'}\left(\mathbb{R}^{d}\right)$, the following facts
hold.
\begin{enumerate}[label=(\roman*)]
\item  $\mathcal{B}_{A}\left(f,g\right)\in L^{q}\left(\rdd\right)$,
with
\[
\left\Vert \mathcal{B}_{A}\left(f,g\right)\right\Vert _{q}\le\frac{\left\Vert f\right\Vert _{p}\left\Vert g\right\Vert _{p'}}{\left|\det\left(M+\frac{1}{2}I\right)\right|^{\frac{1}{p}-\frac{1}{q}}\left|\det\left(M-\frac{1}{2}I\right)\right|^{\frac{1}{p'}-\frac{1}{q}}}.
\]
In particular, $\mathcal{B}_{A}\left(f,g\right)$ is bounded $\left(q=\infty\right)$.
\item $\mathcal{B}_{A}\left(f,g\right)\in C_{0}\left(\rdd\right)$. 
\end{enumerate}
\end{theorem}
We remark that for MWDs in the Cohen's class, a number of these properties
still hold under the weaker assumption $M\in\GLL$.
This is in fact a consequence of convolving with a bounded kernel $\theta_{M}\in L^{\infty}\left(\rdd\right)$.


\section*{Acknowledgements}
The authors would like to thank the anonymous referees for the careful review and the constructive comments, which definitely helped to improve the readability and the quality of the manuscript.

\end{document}